\documentclass[11pt,reqno]{amsart}
\usepackage{hyperref}
\usepackage{amssymb,amsmath,amsfonts,latexsym}
\usepackage{bm}
\usepackage{mathtools,mathrsfs}
\usepackage{enumerate}
\usepackage{array,graphics,color}
\allowdisplaybreaks

\numberwithin{equation}{section}
\newtheorem{dfn}{Definition}[section]
\newtheorem{thm}[dfn]{Theorem}
\newtheorem{lma}[dfn]{Lemma}

\newtheorem{ppsn}[dfn]{Proposition}
\newtheorem{crlre}[dfn]{Corollary}

\newtheorem{rmrk}[dfn]{Remark}

\DeclarePairedDelimiterX{\norm}[1]{\lVert}{\rVert}{#1}
\DeclarePairedDelimiterX{\bnorm}[1]{\big\lVert}{\big\rVert}{#1}
\DeclarePairedDelimiterX{\Bnorm}[1]{\Big\lVert}{\Big\rVert}{#1}
\newcommand\at[2]{\left.#1\right|_{#2}}

\newcommand{\R}{\mathbb{R}}

\newcommand{\Nat}{\mathbb{N}}

\newcommand{\hil}{\mathcal{H}}

\newcommand{\blm}{\mathcal{B}(\ell_2^m)}

\newcommand{\Tr}{\operatorname{Tr}}
\newcommand{\sgn}{\operatorname{sgn}}

\newcommand{\la}{\langle}
\newcommand{\ra}{\rangle}

\begin{document}

	\title[Isometric Embeddability of $S_q^m$ into $S_p^n$]{Isometric Embeddability of $S_q^m$ into $S_p^n$}
	\author[Chattopadhyay] {Arup Chattopadhyay}
	\address{Department of Mathematics, Indian Institute of Technology Guwahati, Guwahati, 781039, India}
	\email{arupchatt@iitg.ac.in, 2003arupchattopadhyay@gmail.com}
	
	\author[Hong] {Guixiang Hong}
	\address{School of Mathematics and Statistics, Wuhan university, Wuhan-430072, China}
	\email{guixiang.hong@whu.edu.cn}
	
	\author[Pal] {Avijit Pal}
	\address{Department of Mathematics, Indian Institute of Technology, Bhilai-492015, India}
	\email{avijit@iitbhilai.ac.in}
	
	\author[Pradhan] {Chandan Pradhan}
	\address{Department of Mathematics, Indian Institute of Technology Guwahati, Guwahati, 781039, India}
	\email{chandan.math@iitg.ac.in, chandan.pradhan2108@gmail.com}
	
	\author[Ray] {Samya Kumar Ray}
	\address{School of Mathematics and Statistics, Wuhan University, Wuhan-430072, China}
	\email{samyaray7777@gmail.com}

	
	
	%\extraline{The first named author thankfully acknowledges the financial support provided by Mathematical Research Impact Centric Support (MATRICS) grant, File no: MTR/2019/000640, by the Science and Engineering Research Board (SERB), Department of Science and Technology (DST), Government of India. The second named author is supported by NSF of China ($\#$12071355). The third named author acknowledges Serb-DST(ECR/2018/002243 dated 12 March 2019) for financial support. The fourth named author thanks the Indian Institute of Technology Guwahati, Government of India, for the financial support. The fifth named author acknowledges Key-subsidy postdoctoral fellowship supported by Wuhan University (NSF (China), Grant-12071355).}

	\begin{abstract}
		In this paper, we study existence of isometric embedding of $S_q^m$ into $S_p^n,$ where $1\leq p\neq q\leq \infty$ and $n\geq m\geq 2.$ We show that for all $n\geq m\geq 2$ if there exists a linear isometry from $S_q^m$ into $S_p^n$, where $(q,p)\in(1,\infty]\times(1,\infty) \cup(1,\infty)\setminus\{3\}\times\{1,\infty\}$ and $p\neq q,$ then we must have $q=2.$ This mostly generalizes a classical result of Lyubich and Vaserstein. We also show that whenever $S_q$ embeds isometrically into $S_p$ for $(q,p)\in \left(1,\infty\right)\times\left[2,\infty \right)\cup[4,\infty)\times\{1\} \cup\{\infty\}\times\left( 1,\infty\right)\cup[2,\infty)\times\{\infty\}$ with $p\neq q,$ we must have $q=2.$ Thus, our work complements work of Junge, Parcet, Xu and others on isometric and almost isometric embedding theory on non-commutative $L_p$-spaces. Our methods rely on several new ingredients related to perturbation theory of linear operators, namely Kato-Rellich theorem, theory of multiple operator integrals and Birkhoff-James orthogonality, followed by thorough and careful case by case analysis. The question whether for $m\geq 2$ and $1<q<2,$ $S_q^m$ embeds isometrically into $S_\infty^n$, was left open in \textit{Bull. London Math. Soc.} 52 (2020) 437-447. 
	\end{abstract}
	\maketitle

\section{Introduction and main results}
The isometric theory of Banach spaces have a long history. In his seminal work \cite{Ba32}, Banach characterized isometries from $\ell_p$ to $\ell_p$ for $1\leq p\neq 2<\infty.$ This work was taken up further by Lamperti \cite{La58} who characterized isometries from $L^p(\Omega)$ to $L^p(\Omega),$ for $1\leq p\neq 2<\infty,$ where $\Omega$ is a $\sigma$-finite measure space. We refer \cite{FlJa03}, \cite{FlJa07} for a detailed exposition on this topic. However, embeddings between different $L_p$-spaces also have a long history and led to strong interactions between Banach space theory, probability theory, geometry of convex bodies, harmonic analysis and combinatorics. The work of Le\'vy \cite{Le37} and Schoenberg \cite{Sc38} made strong connections between positive definite functions, stable random variables and isometric embeddings of Banach spaces into $L_p$-spaces.  Rosenthal \cite{Ro73} obtained structural results on subspaces of $L_p([0,1])$ for $1\leq p<2.$ Delbaen et al. \cite{DeJaPe98} established a characterization of subspaces of $L_p([0,1])$ which embed isometrically into  $\ell_p$ or $\ell_p^n,$ where $0<p<\infty,$ answering a question of Albrecht Pietsch. It is also known that \cite{BrDiKr66,He71} $L_q([0,1])$ isometrically embeds in $L_p([0,1]),$ where $0<p<q\leq 2.$ Isometric embeddings between finite dimensional $\ell_p$-spaces were also studied. In \cite{LySh04}, the authors showed that if there exists a linear isometry from $\ell_q^m$ to $\ell_p^n,$ where $1\leq p\neq q\leq \infty,$ and $2\leq m\leq n,$ then we must have $q=2$ and $p$ to be an even integer. Moreover, given $m\in\mathbb N$ and $p\in 2\mathbb N,$ they also established quantitative estimates on $n\in\mathbb N,$ depending on $m$ and $p,$ for which $\ell_2^m$ embeds isometrically into $\ell_p^n.$ The origin of these kind of results go back to \cite{LyVa93}, where similar problems were studied in the context of finite dimensional $\ell_p$-spaces over $\mathbb R$. Remarkably, this line of research is related to spherical design, combinatorics and number theory. We refer \cite{Ly08}--\cite{LySh01} for more on this topic. Finally, we mention \cite{KoKo01} and references therein for many results on isometric and almost isometric embeddings between $L_p$ and $\ell_p$-spaces and a historical account on this direction of research.

Inspired by quantum mechanics, non-commutative analysis has been developed very quickly in the past few decades. In this direction of research, quantization of measure theory leads to the theory of non-commutative $L_p$-spaces. Isometries on non-commutative $L_p$-spaces were studied in \cite{Ar75}, \cite{Ye81}, \cite{JuRuSh05}. In this new set-up, one often encounters significant difficulty, and many purely non-commutative phenomenons occur. For example unlike the classical situation \cite{LySh04} $S_2^2$ embeds isometrically into $S_\infty^n$ for any $n\geq 4$ by identifying $S_2^2$ with row or column Hilbert space. However, remarkable progress has been made, and the quest for non-commutative analogs of classical results, some of which were mentioned earlier in the introduction, started a rich interaction between free probability theory, operator spaces, and non-commutative $L_p$-spaces.  For example, to establish a non-commutative analog of Rosenthal's theorem was a long-standing open problem which was finally settled by Junge and Parcet \cite{JuPa082}. Pisier \cite{Pi04} showed that the operator Hilbert space does not embed completely isomorphically into predual of semifinite von Neumann algebras. In \cite{Ju05}, the author established that the operator Hilbert space embeds \textit{completely isomorphically} into predual of a von Neumann algebra of type $\text{III}.$ It follows from work of Xu \cite[Corollary 5.7]{Xu06} that $S_q$ does not embed \textit{completely isomorphically} into a semifinite non-commutative $L_p$-space for $1\leq p<q\leq 2$. On the other hand, $S_q$ embeds \textit{isomorphically} into $L_p(\mathcal R),$ where $\mathcal R$ is a hyperfinite type $\text{II}_1$ factor for $1\leq p<q\leq 2$ \cite{Ju00}. It follows from the main result of \cite{Ju00} that for all $0<p<q<2,$ $L_q(\mathcal R\overline{\otimes}B(\ell_2))$ and hence $L_q(\mathcal R)$ embeds isometrically into $L_p(\mathcal R)$. Moreover, existence of completely isometric embedding of $L_q(\mathcal R)$ into $L_p(\mathcal R)$ is an open problem \cite[P. 10]{JuPa10}. We refer \cite{JuPa10}, \cite{JuPa08}, \cite{JuPa082}, \cite{JuNiRuXu04}, \cite{SuXu03}, \cite{JuRuSh05}, \cite{RaXu03} and references therein for more information in this direction of research.

In this paper, we study \textit{isometric embeddability} between finite-dimensional Schatten-$p$ classes and partially establish a non-commutative analog of the main result of \cite{LySh04}. Thus, our result is another contribution in the direction of non-commutative isometric embedding theory. We also provide similar results for embeddability of infinite-dimensional Schatten-$p$ classes and find interesting implications on non-commutative $L_p$-spaces. We prove the following theorem. 
\begin{thm}\label{main}
	Let $(q,p)\in(1,\infty]\times(1,\infty) \cup(1,\infty)\setminus\{3
	\}\times\{1,\infty\}$ with $p\neq q$ and $2\leq m\leq n$. 
	If there exists an isometric embedding of $S_q^m$ into $S_p^n$,
	then we must have $q=2$.
\end{thm}
The above theorem can be concluded from Theorem \ref{isometric}, Theorem \ref{thm 3.4} and  Theorem \ref{th2}. It was implicitly proved in \cite{GuRe18} that $S_1^m$ does not embed isometrically into $S_\infty^n$. Their proof relied on a clever trick based on unitary dilation of contractions on a Hilbert space. Later, Pisier guessed that the same result is true when $S_\infty^n$ is replaced by $S_\infty$ (see \cite{GuRe18}) and G. Misra asked whether $\ell_q^m$ embeds into $S_\infty^n$ for all $1\leq q\neq 2<\infty$ and $2\leq m\leq n.$ All these questions were settled in \cite{Ra20} using Birkhoff-James orthogonality and norm parallelism, combined with infinite descent type arguments. It was also proved in \cite{Ra20} that $S_1^m$ does not embed isometrically into $S_p$ for all $m\geq 2$ and $1<p<\infty.$ However, the question, whether $\ell_q,$ $q\in(1,2)$ embeds isometrically into $S_\infty^n$ was left open in \cite{Ra20}.
% In this article, we prove that for $(q,p)\in(1,\infty]\times(1,\infty)\cup(1,2)\times\{\infty\}$ with $p\neq q$ and $2\leq m\leq n,$ if there exists an isometric embedding of $\ell_q^m$ into $S_p^n$,
%then we must have $q=2$. From this Theorem \ref{main} can be derived at once.
Thus in Theorem \ref{main}, we settle this remaining case, by giving yet another proof which is completely different from \cite{Ra20} and new even in the commutative case. However, we do not know if $S_2^m$ embeds isometrically into $S_p^n$ where $1<p<\infty$ and $p$ is not an even integer. If $p\in 2\mathbb N$ and $m\in\mathbb N$ are given, then it can be easily seen from \cite[Theorem B]{LySh04} that $S_2^m$ embeds isometrically into $S_p^n$ for any $n\in\mathbb N,$ which is sufficiently large. Moreover, we also do not know whether $S_q^m$ embeds isometrically into $S_1^n,$ where $n\geq m\geq 2$ and $q\in\{2,3\}$.

Let us briefly mention some key steps and new ingredients used in the proof of Theorem \ref{main} when $1<p\neq q<\infty$.
\begin{itemize}
	\item[(i)]\textit{Multiple operator integral:} In \cite{LyVa93}, a solution for the classical isometric embedding problem between finite dimensional $\ell_p$-spaces was given, where the underlying scalar field was the field of real numbers. In our situation, we reduce the problem to Hermitian case, by using a $2\times 2$ trick. However, we observe that the proof given in \cite{LySh04} is more adaptable. Following \cite{LySh04}, we are reduced to showing that $\at{\frac{d^2}{dt^2}}{t=0}\|A+tB\|_p^p\neq 0$, where $(z,w)\mapsto zA+wB$ is an isometry from $\ell_q^2$ into $S_p^n.$ The formulas for derivatives of Schatten norms were given in \cite{SkTo19} and \cite{PoSu14} in terms of multiple operator integrals. By using these formulas, we calculate $\at{\frac{d^2}{dt^2}}{t=0}\|A+tB\|_p^p$ in terms of coefficients of the matrices $A$ and $B$ and show that this is non-zero. This calculation is difficult in the non-commutative setting; for instance, in the particular case $1<p<2$ and $A$ invertible, the precise calculation of the second derivative is the main contribution of \cite{RiXu16} and \cite{BaCaLi94}, which they handled with different methods than ours. In our case, the situation is much more complicated since $A$ may not be invertible, and we also handle the case $p>2.$ Also, our approach depends partially on the Kato-Rellich theorem.
	
	We would like to mention that recently Zhang \cite{Zh20} used the methods of our work to provide an alternative proof of a celebrated theorem of Ball, Carlen, and Lieb \cite{BaCaLi94}.
	\item[(ii)]\textit{Kato-Rellich theorem:} One of the main ingredients of our proof is a remarkable theorem of Rellich and Kato on perturbation theory of linear operators, which allows a local real-analytic choice of a complete set of eigenvalues and eigenvectors of any one-parameter real-analytic family of self-adjoint operators. This enables us to use various facts on real-analytic functions and use asymptotic comparisons around the neighborhood of the origin. This approach is new, even in the commutative case. Moreover, by using Kato-Rellich theorem, we give a new method to calculate $\at{\frac{d}{dt}}{t=0}\|A+tB\|_p^p.$ Previously, this was done in \cite{Ko84}.
\end{itemize}
We combine the above tools with Birkhoff-James orthogonality to obtain our result by delicate case by case analysis. For other values of $q$ and $p$ our analysis relies mostly on the Kato-Rellich theorem. It is natural to approach our problem by various invariants known in Banach space geometry such as modulus of convexity, type, cotype etc. But those concepts are more suited for infinite dimensional analysis and are not much help in our situation. However, to study the case $(q,p)\in[4,\infty)\times\{1\}$, we use $PL$-convexity of $S_1$, which was originally observed by Haagerup. We also prove an infinite-dimensional version of Theorem \ref{main}. 
%and obtain some consequences on general noncommutative $L_p$-spaces.
This time the main dificulty lies on showing $\at{\frac{d^2}{dt^2}}{t=0}\|A+tB\|_p^p\neq 0$ but rest of the proof is verbatim. However, due to lack of an appropriate version Kato-Rellich theorem, we only obtain our result for $(q,p)\in \left(1,\infty\right)\times\left[2,\infty \right)\cup[4,\infty)\times\{1\} \cup\{\infty\}\times\left( 1,\infty\right)$.

We end the introduction by briefly mentioning the organization of the paper. In Section \ref{prel}, we describe all the necessary preliminaries and notations and prove some elementary lemmas. In Section \ref{endpoint} we prove Theorem \ref{isometric} and Theorem \ref{thm 3.4}. In Section \ref{isofinite}, we give the proof of Theorem \ref{th2}. In Section \ref{isoinfinite}, we prove infinite-dimensional analogs of our results.

\section{Preliminaries}\label{prel}
Let $\mathcal M$ be a von Neumann algebra with a normal, faithful, semifinite trace $\tau.$ Let $\mathcal{S}(\mathcal{M})$ be the linear span of set of all positive elements in $\mathcal M$ such that $\tau(s(x))<\infty,$ where $s(x)$ denotes the support projection of $x$. The non-commutative $L_p$-space is the closure of $\mathcal{S}(\mathcal{M})$ with respect to the norm $\|x\|_p:=(\tau(|x|^p))^{\frac{1}{p}}$ and is denoted by $L_p(\mathcal M),$ where $1\leq p<\infty.$ Denote $L_\infty(\mathcal M)=\mathcal M.$ When $\mathcal M=\mathcal{B}(\hil)$ and $\tau$ is the usual trace $Tr$ on $\mathcal{B}(\hil),$ the associated non-commutative $L_p$-spaces are known to be Schatten-$p$ classes for $1\leq p<\infty$ and denoted by $S_p(\mathcal H).$ We denote $S_\infty(\mathcal H)$ to be the set of all compact operators. Whenever $\mathcal H$ is $\ell_2^n$ or $\ell_2,$ we denote $S_p(\mathcal H)$ by $S_p^n$ or $S_p$ respectively. The set of all $n\times n$ complex matrices is denoted by $M_n$. We denote $M_n^{sa}$, $S_p^{sa}(\mathcal{H}),$ $S_p^{sa}$ and $L_p^{sa}(\mathcal M)$ to be the set of all self-adjoint elements in the corresponding spaces. 
\subsection{Kato-Rellich theorem and Birkhoff-James orthogonality:}
The following theorem can be deduced from a celebrated theorem of Kato and Rellich (see \cite[Page 31, Chapter I]{ReBe69},\cite[P. 21, Theorem 1]{Ba85}). For simplicity, we state the theorem for the case of real-analytic family of self-adjoint operators of the form $A+tB,$ where $A,B\in M_n^{sa}$ and $t\in\mathbb R.$
\begin{thm}\label{kato} Let $A,B\in M_n^{sa}.$ Then, for all $t_0\in\mathbb R$, there exist $\epsilon(t_0)>0$ and real-analytic functions $\lambda_k:(t_0-\epsilon(t_0),t_0+\epsilon(t_0)) \to \mathbb R,$ where $1\leq k\leq n$ and real-analytic functions $u_{ij}:(t_0-\epsilon(t_0),t_0+\epsilon(t_0)) \to \mathbb R,$ where $1\leq i,j\leq n$ such that
	\begin{itemize}
		\item[(i)]for all $t\in(t_0-\epsilon(t_0),t_0+\epsilon(t_0))$, $\{\lambda_1(t),\dots,\lambda_n(t)\}$ is the complete set of eigenvalues of $A+tB$ counting multiplicity.
		\item[(ii)] for all  $t\in(t_0-\epsilon(t_0),t_0+\epsilon(t_0))$, $U(t):=(u_{ij}(t))_{i,j=1}^n$ is a unitary matrix.
		\item[(iii)] for all  $t\in(t_0-\epsilon(t_0),t_0+\epsilon(t_0))$, $U(t)^*(A+tB)U(t)=\text{diag}(\lambda_1(t),\dots,\lambda_n(t)).$ 
	\end{itemize}
\end{thm}
We refer \cite{Ka13}, \cite{ReBe69} and \cite{Ba85} for more general version of Theorem \ref{kato} and applications. However, if $A$ or $B$ is a operator acting on an infinite dimensional Hilbert space, the exact analogue of the above theorem is not true.

We recall the notion of Birkhoff-James orthogonality, which was introduced in \cite{BI35} and \cite{JA45}. Let $\mathcal X$ be a Banach space. For $A,B\in B(\mathcal X),$ we say that $A$ is orthogonal to $B$ in the sense of Birkhoff-James if $\|A+zB\|_{\mathcal X\to\mathcal X}\geq\|A\|_{\mathcal X\to\mathcal X}$ for all complex numbers $z\in\mathbb C.$ In this case, we write $A\perp_{BJ}B.$ We have the following characterization of Birkhoff-James orthogonality, which we state for self-adjoint operators in Schatten class.
%\noindent Let $1\leq p<\infty$. Next we define the semi-inner product $[\cdot,\cdot]:S_p\times S_p\to \R$ as follows
%\[[A,B]=\norm{A}_p^{2-p}\Tr(|A|^{p-1}U^*B),~A,B\in S_p,\] where $U$ is the partial isometry corresponding to the polar decomposition of $A$. The following result gives a connection between the above semi-inner product and the Birkhoff-James orthogonality corresponding to the operators $A$ and $B$.  
\begin{thm}\cite[Theorem 3.8]{BoCoMoWoZa19}\label{l3}
	Let $1<p<\infty$. For $A,B\in S_p^{sa},$ the following statements are equivalent:
	\begin{align*}
		(i)~& A\perp_{BJ}B.\hspace*{5in}\\
		(ii)~& \Tr(|A|^{p-1}UB)=0,
	\end{align*} where $U$ is the partial isometry corresponding to the polar decomposition of $A$.
\end{thm}
\begin{rmrk}\label{1strem}
	In the above theorem, $(i)\implies(ii)$ is true even in general non-commutative $L_p$-spaces, where $1<p<\infty$. Just observe that if $A,B\in L_p^{sa}(\mathcal M)$ with $A\perp_{BJ}~B,$ then the map $t\mapsto \|A+tB\|_p^p$ attains minimum at $t=0.$ Therefore, we have $\at{\frac{d}{dt}}{t=0}\|A+tB\|_p^p= 0.$ Now it follows from \cite[Lemma 3.1.]{Ko84} (also see \cite[Lemma 5.8.4.]{SkTo19}) that $\tau(|A|^{p-1}UB)=0$, $U$ being the partial isometry in the polar decomposition of $A$.
\end{rmrk}
Following Remark \ref{1strem}, we have the following corollary which seems to be new even in the commutative case. Denote $(e_1,e_2)$ to be the standard basis of $\mathbb{C}^2.$
\begin{crlre}
	Let $1<p<\infty$ and $1\leq q\leq \infty.$ If $p\neq q,$ there is no linear isometry $T:\ell_q^2\to L_p(\mathcal M)$ with $T(e_1)=A,$ $T(e_2)=B$ such that $A\geq 0$ and $B\geq 0$.
\end{crlre}
\begin{proof}Since $A\geq 0,$ we have $|A|^{p-1}U=A^{p-1}U=A^{p-1},$ $U$ being the partial isometry in the polar decomposition of $A$. Note that as $T:\ell_q^2\to L_p(\mathcal M)$ is a linear isometry, we must have that $\|T(e_1)\|_p=\|A\|_p=1.$ Similarly, $\|B\|_p=1.$ Therefore, we have that \[\|A+zB\|_p=\|(1,z)\|_q\geq 1=\|A\|_p\] for all $z\in\mathbb C.$ Therefore, $A\perp_{BJ}B$ and by Remark \ref{1strem}, we have 
	\begin{equation*}
		\tau(A^{p-1}B)=\tau(B^{\frac{1}{2}}A^{p-1}B^{\frac{1}{2}})=0.
	\end{equation*} Hence by faithfulness of $\tau,$ we have 
\begin{align}\label{eqn 2.1}
	B^{\frac{1}{2}}A^{\frac{p-1}{2}}=0.
\end{align}
 Let $A=\int \lambda dE_A(\lambda)$ and $B=\int \lambda dE_B(\lambda)$ be the spectral resolutions of $A$ and $B$ respectively. Denote $A_1:=\int \lambda^{\frac{-p+1}{2}} \chi_{(0,\infty)}dE_A(\lambda)$ and $B_1:=\int\lambda^{-\frac{1}{2}} \chi_{(0,\infty)}dE_B(\lambda).$ Multiplying from left by $B_1$ and right by $A_1$ in \eqref{eqn 2.1}, we clearly have $s(B)s(A)=0.$ Then, it follows that  \[\|zA+wB\|_p=(|z|^p+|w|^p)^{\frac{1}{p}}=(|z|^q+|w|^q)^{\frac{1}{q}}\] for all $z,w\in\mathbb C,$ where the first equality follows from (see \cite[Proof of Theorem 1]{Ye81}) and the second equality follows from the fact that $T$ is a linear isometry. Thus we must have $p=q.$
\end{proof}	
\subsection{Multiple operator integrals and operator derivatives in $\ell_2^m$}
Let $A_0,A_1,\ldots,A_n$ be self-adjoint operators defined on an $m$-dimensional Hilbert space $\ell_2^m$.
Let $e^{(j)}=\{e_i^{(j)}\}_{i=1}^{m}$ be an orthonormal basis of eigenvectors of $A_j$ and let $\lambda^{(j)}=(\lambda_i^{(j)})_{i=1}^{m}$ be the corresponding $m$-tuple of eigenvalues, $j=0,1,\ldots,n$. Let $\phi:\R^{n+1}\to\mathbb{C}$. Now one defines $T^{A_0,A_1,\ldots,A_n}_\phi$: $\underbrace{\blm\times\blm\times\cdots\times\blm}_{\text{n-times}}\longrightarrow \blm$ by 
\begin{align}\label{eq1}
	T^{A_0,A_1,\ldots,A_n}_\phi(B_1,B_2,\ldots,B_n)=\sum_{i_0,i_1,\ldots,i_n=1}^{m}\phi(\lambda_{i_0}^{(0)},\lambda_{i_1}^{(1)},\ldots,\lambda_{i_n}^{(n)})P_{e_{i_0}^{(0)}}B_1P_{e_{i_1}^{(1)}}B_2\cdots B_nP_{e_{i_n}^{(n)}},
\end{align}
for $B_1,B_2,\ldots,B_n\in\blm$ and $P_{e_{i_k}^{(k)}}$ is the orthogonal projection of $\ell_2^m$ onto the subspace spanned by the vector $\{e_{i_k}^{(k)}\}$. Note that \eqref{eq1} can also be rewritten in terms of the spectral measures $E_{A_j}$ of $A_j$. Assume that $\{\lambda_i^{(j)}\}_{i=1}^{d_j}$ is the set of pairwise distinct eigenvalues of the operator $A_j$, where $d_j\in\Nat$ and $d_j\leq m$. Then we have the following equivalent form of \eqref{eq1}.
\begin{align*}
	& T^{A_0,A_1,\ldots,A_n}_\phi(B_1,B_2,\ldots,B_n)\\
	& =\sum_{i_0=1}^{d_0}\cdots\sum_{i_n=1}^{d_n}\phi(\lambda_{i_0}^{(0)},\lambda_{i_1}^{(1)},\ldots,\lambda_{i_n}^{(n)})E_{A_0}(\{\lambda_{i_0}^{(0)}\})B_1E_{A_1}(\{\lambda_{i_1}^{(1)}\})B_2\cdots B_nE_{A_n}(\{\lambda_{i_n}^{(n)}\}).
\end{align*} The operator $T^{A_0,A_1,\ldots,A_n}_\phi$ is a discrete version of a multiple
operator integral. The function $\phi$ is usually called the symbol of the operator $T^{A_0,A_1,\ldots,A_n}_\phi$. For more about discrete multiple operator integral see \cite[Chapter 4]{SkTo19}.\\
Throughout this paper we use the notation $C^k(\R)$, for the space of all $k$-times differentiable functions with continuous $k$-th derivative, where $k\in\mathbb N\cup\{0\}$.
Let $f\in C^r(\R)$. Recall that the divided difference of order $r$ is an operation on the function $f$ of one
(real) variable, and is defined recursively as follows:
\begin{align*}
	f^{[0]}(\lambda)&:=f(\lambda),\\
	f^{[r]}(\lambda_0,\lambda_1,\ldots,\lambda_r)&:=\begin{cases*}
		\frac{f^{[r-1]}(\lambda_0,\lambda_1,\ldots,\lambda_{r-2},\lambda_r)-f^{[r-1]}(\lambda_0,\lambda_1,\ldots,\lambda_{r-2},\lambda_{r-1})}{\lambda_r-\lambda_{r-1}} \quad \text{if}\quad \lambda_r\neq\lambda_{r-1},\\
		\at{\frac{\partial}{\partial \lambda}}{\lambda=\lambda_{r-1}}f^{[r-1]}(\lambda_0,\lambda_1,\ldots,\lambda_{r-2},\lambda)\quad \text{if}\quad \lambda_r=\lambda_{r-1}.
	\end{cases*}
\end{align*}
%Let $M_n(\mathbb{C})$ denote the collection of all $n\times n$ complex matrices.

\begin{thm}\cite[Theorem 5.3.2]{SkTo19}\label{th1}
	Let $A,B\in M_n^{sa}$ and $f\in C^r(\R), r\in\Nat$. Then the function $\mathbb{R} \ni t\mapsto f(A+tB)$ is $r$-times differentiable in the operator norm and the $r$-th order derivative is given by the formula
	\begin{align}
		\label{eq2}&\frac{1}{r!}\at{\frac{d^r}{dt^r}}{t=0}f\big(A+tB \big)=T^{A,A,\ldots,A}_{f^{[r]}}\underbrace{(B,B,\ldots,B)}_{r\text{-times}}
	\end{align}
	and hence 
	\begin{align}
		\label{eq3}&\frac{1}{r!}\at{\frac{d^r}{dt^r}}{t=0}\Tr\big(f\big(A+tB \big)\big)=\Tr\bigg(T^{A,A,\ldots,A}_{f^{[r]}}\underbrace{(B,B,\ldots,B)}_{r\text{-times}}\bigg).
	\end{align}
\end{thm}
However, we give the following proof of the formula of calculating the first derivative by using the Kato-Rellich theorem. Our proof is completely different from \cite{Ko84}. For a self-adjoint matrix $A\in M_n^{sa},$ the partial isometry appearing in the polar decomposition is denoted by $\text{sgn}(A).$
\begin{ppsn}
	Let $A,B\in M_n^{sa}$ and $1<p<\infty$, then 
	\begin{align}
		\at{\frac{d}{dt}}{t=0}\norm{A+tB}_p^p=p\Tr\big(B|A|^{p-1}\sgn(A)\big).
	\end{align}
\end{ppsn}
\begin{proof}
	By Theorem \ref{kato}, there exist unitary matrix $U(t)$
	and a diagonal matrix $\lambda(t)$ such that $A+tB=U(t)^*\lambda(t)U(t)$ and $U(t),\lambda(t)$ are real-analytic on a neighborhood of $0$. Since $\|.\|_p$ is unitary invariant, without loss of any generality, we assume that $A$ is a real diagonal matrix and $U(0)=I$. By differentiating, the identity $U(t)^*U(t)=U(t)U(t)^*=I,$ we obtain
	\begin{equation}\label{makinze} {U^*}'(0)+U'(0)=0.
	\end{equation} Applying the standard functional calculus on $A+tB$ and differentiating, we have
	\begin{align*}
		\at{\frac{d}{dt}}{t=0}|A+tB|^p=&{U^*}'(0)|\lambda(0)|^pU(0)\\
		&+U^*(0)|\lambda(0)|^pU'(0)+pU^*(0)|\lambda(0)|^{p-1}\sgn{\lambda(0)}\lambda'(0)U(0).
	\end{align*} 
	Taking trace on both side of the above equation and applying \eqref{makinze}, we have
	\begin{align}\label{a9}
		\nonumber\Tr\big(\at{\frac{d}{dt}}{t=0}|A+tB|^p\big)=&\Tr\big({U^*}'(0)|\lambda(0)|^p
		+|\lambda(0)|^pU'(0)\big)+p\Tr\big(|\lambda(0)|^{p-1}\sgn{\lambda(0)}\lambda'(0)\big)\\
		\nonumber=&\Tr\big(({U^*}'(0) +U'(0))|\lambda(0)|^p\big)+p\Tr\big(\lambda'(0)|A|^{p-1}\sgn(A)\big)\\
		=&~p\Tr\big(\lambda'(0)|A|^{p-1}\sgn(A)\big).
	\end{align}
	Note that
	\begin{align}\label{a10}
		\nonumber&B=\at{\frac{d}{dt}}{t=0}(A+tB)={U^*}'(0)\lambda(0)+\lambda(0)~U'(0)+{\lambda'(0)}\\
		\implies ~&\lambda'(0)=B-({U^*}'(0)\lambda(0)+\lambda(0)~U'(0)).
	\end{align}
	
	Therefore from (\ref{a9}) and (\ref{a10}), and using the fact that $A,$ $\sgn(A)$ and $|A|$ are mutually commuting and \eqref{makinze}, we have
	\begin{align*}
		\at{\frac{d}{dt}}{t=0}\norm{A+tB}_p^p=&~p\Tr\big(\lambda'(0)|A|^{p-1}\sgn(A)\big)\\
		=&~p\Tr\Big\{ B|A|^{p-1}\sgn(A)- \left({U^*}'(0)\lambda(0)+\lambda(0)U'(0)\right)|A|^{p-1}\sgn(A)\Big\}\\
		=&~p\Tr\left( B|A|^{p-1}\sgn(A)\right)-p\Tr\bigg(\left({U^*}'(0)+U'(0)\right)A|A|^{p-1}\sgn(A) \bigg)\\
		=&~p\Tr( B|A|^{p-1}\sgn(A)).
	\end{align*}
	This completes the proof of the proposition.
\end{proof}
%\begin{lma}\label{l1}
%Let $A$ and $B$ be two self-adjoint operators in a finite dimensional Hilbert space $\hil$. If $A+tB$ ($t\in \mathbb{R}$) is invertible in a neighborhood of origin (say $\Delta$), then there exists a positive real number $\alpha$ such that the spectrum of $A+tB$ near the point $0$ is supported in the set $[-\beta,-\alpha]\cup[\alpha,\beta]$, that is $\sigma(A+tB)\subseteq [-\beta,-\alpha]\cup[\alpha,\beta]$ for $t\in \Delta$, where $\beta=\norm{A}+\norm{B}$.
%\end{lma}

%{\color{red}\noindent Note that the above result is also true for infinite dimension as well.}
\subsection{Reduction to self-adjoint case:} We now state the following lemma, which reduces our problem to the self-adjoint case.
\begin{lma}\label{reduct}
	Let $1\leq p,q\leq \infty.$ Let $T:\ell_q^2\to S_p(\mathcal H)$ with $T(e_1)=A,$ $T(e_2)=B$ be an isometry. Then \[T_{new}:\ell_q^2\to S_p(\mathcal H\oplus\mathcal H)\text{ defined by } T_{new}(z,w)\colon=zA_{new}+wB_{new}\] is again an isometry, where \[A_{new}(\zeta_1\oplus\zeta_2)\colon=2^{-\frac{1}{p}}\Big(A\zeta_2\oplus A^*\zeta_1\Big)~ \text{and}~ B_{new}(\zeta_1\oplus\zeta_2)\colon=2^{-\frac{1}{p}}\Big(B\zeta_2\oplus B^*\zeta_1\Big).\]
\end{lma}
\begin{proof}
	Note that for all $(z,w)\in \mathbb{C}^2,$ we have
	
	\[T_{new}(z,w)=2^{-\frac{1}{p}}\begin{bmatrix}
		0&zA+wB\\
		zA^*+wB^*&0
	\end{bmatrix}.\] Since Schatten norms are invariant under multiplication by unitary operators, by changing columns, we have from above
	\[\|T_{new}(z,w)\|_p=2^{-\frac{1}{p}}\Big\|\begin{bmatrix}
		zA+wB&0\\
		0&zA^*+wB^*
	\end{bmatrix}\Big\|_p=2^{-\frac{1}{p}}(\|zA+wB\|_p^p+\|zA+wB\|_p^p)^{\frac{1}{p}}.\] Therefore, we have that
	\[\|T_{new}(z,w)\|_p=2^{-\frac{1}{p}}(2\|(z,w)\|_q^p)^{\frac{1}{p}}=\|(z,w)\|_q.\]
	This completes the proof of the lemma.
\end{proof}
\begin{dfn}\label{hp}
	Let $1\leq p\neq q\leq \infty$. Let $A,B\in \mathcal{B}(\ell_2)$ or $\mathcal{B}(\ell_2^n).$ We say $(A,B)$ has $(\mathbf{I}_{q,p})$ if
	\begin{itemize}
		\item[$\bullet$] with respect to the standard basis, $A$ is a diagonal matrix with all entries real and $B$ is self-adjoint,
		\item[$\bullet$] $\norm{A}_p=\norm{B}_p=1$,
		\item[$\bullet$] $\norm{A+tB}_p=\|(1,t)\|_q$, for all $t\in\R$.
	\end{itemize}
\end{dfn}
Note that, if we have an isometric embedding of $S^m_q$ into $S_p(\mathcal H),$ with $\mathcal H=\ell_2$ or $\ell_2^n$, $ 1\leq p\neq q\leq \infty$ and $n\geq m\geq 2,$ then by Lemma \ref{reduct} we can find a pair of self-adjoint operators $(A,B)$ where $A, B\in B(\mathcal H\oplus_2\mathcal H)$ and $(A,B)$ has $(\mathbf{I}_{q,p})$. Therefore, nonexistence of isometric embedding will follow from nonexistence of $(A,B)$ having $(\mathbf{I}_{q,p})$. This fact will be crucially used in the rest of the paper.\\
\par The following lemma will be useful in later sections. 
\begin{lma}\label{l}
	Let $1<p\neq q<\infty$ and $(A,B)$ has $(\mathbf{I}_{q,p})$. %$A,B\in \mathcal{B}_p(\mathcal{H})$ satisfying $\|A\|_p=\|B\|_p=1$ and $(1+|t|^q)^{p/q}=\norm{A+tB}_p^p$ for all $t\in \mathbb{R}$. 
	Then $AB\neq 0$. 
\end{lma}
\begin{proof}
	On contrary, we assume that $AB=0$ and therefore $AB=BA=0$. Thus $A$ and $B$ are simultaneously diagonalizable and hence $\norm{A+tB}_p^p = 1+|t|^p$ because of the fact that $AB=0$ and $\|A\|_p=\|B\|_p=1$. This implies that 
	\[(1+|t|^q)^{1/q}=(1+|t|^p)^{1/p} \quad \text{for all}\quad t\in\R,\] which is not possible since $p\neq q$. This completes the proof. 
\end{proof} 

\section{Embeddability of $S_q^m$ into $S_\infty^n$:}\label{endpoint}
%Throughout this section, we make the assumption that $q\in(1,\infty)\setminus\mathbb N.$ 
Let $q\in(1,\infty).$ Let $f_{q}:\mathbb R \to \mathbb R$ be the function defined by $$f_{q}(t)=\Big(1+|t|^q\Big)^{\frac{1}{q}}.$$ In the following lemma, we want to compute the derivative of $f_{q}$ with respect to $t.$
\begin{lma}\label{derivative of f} Let  $q\in(1,\infty)\setminus\mathbb N$ and $n\in\mathbb N$ be an integer. Then, there exists a smooth function $g_n:(-\infty,\infty)\to\mathbb R$ with $g_n(0)\neq 0,$ such that
	\begin{itemize}
		\item[(i)] for all $t\in(0,\infty),$ we have \[f_q^{(n)}(t)=t^{q-n}(1+t^q)^{\frac{1}{q}-n}g_n(t^q),\]
		\item[(ii)] and for all $t\in(-\infty,0),$ we have \[f_q^{(n)}(t)=(-1)^n(-t)^{q-n}(1+(-t)^q)^{\frac{1}{q}-n}g_n((-t)^q).\]
	\end{itemize}
\end{lma}
\begin{proof} We only prove (i). One can easily deduce (ii) from (i). We proceed by induction. Note that we have 
	\begin{equation*}
		f_q^{(1)}(t)=\frac{1}{q}(1+t^q)^{\frac{1}{q}-1}qt^{q-1}=t^{q-1}(1+t^q)^{\frac{1}{q}-1}.
	\end{equation*}
	By denoting $g_1\equiv 1,$ we have the required form of $f_q^{(1)}.$ Suppose the hypothesis is true for some $n\in\mathbb N.$ Then, we have 
	\begin{equation*}
		\begin{split}
			f_q^{(n+1)}(t)&=(q-n)t^{q-n-1}(1+t^q)^{\frac{1}{q}-n}g_n(t^q)\\
			&+t^{q-n}(\frac{1}{q}-n)(1+t^q)^{\frac{1}{q}-n-1}qt^{q-1}g_n(t^q)\\
			&+t^{q-n}(1+t^q)^{\frac{1}{q}-n}g_n^{\prime}(t^q)qt^{q-1}.
		\end{split}
	\end{equation*} 	
	Therefore, from the above identity we have
	\begin{equation*}
		f_q^{(n+1)}(t)=t^{q-n-1}(1+t^q)^{\frac{1}{q}-n-1}\Big((q-n)(1+t^q)g_n(t^q)+t^q(1-nq)g_n(t^q)+qt^q(1+t^q)g_n^{\prime}(t^q)\Big).
	\end{equation*}
	If we denote \[g_{n+1}(t):=(q-n)(1+t)g_n(t)+t(1-nq)g_n(t)+qt(1+t)g_n^{\prime}(t),\] then clearly, $g_{n+1}$ is smooth and $g_{n+1}(0)=(q-n)g_n(0)\neq 0$ as $q\notin\mathbb N.$ This completes the proof of the lemma.
\end{proof}	
For $q\geq 1,$ let $g, h:\mathbb C \to \mathbb C$ be two complex (may be multi-valued) functions defined by $$g(z)=1+z^q~~{\rm{ and}}~~~h(z)=(1+z^q)^{\frac{1}{q}}.$$
\begin{ppsn}\label{analytic prom}
	Let $q\in(1,\infty).$ Then the map $f_{q}$ is real-analytic on $(0,\infty)$ and $(-\infty,0).$
\end{ppsn}
\begin{proof}
	We only prove that $f_q$ is real-analytic on $(0,\infty).$ Let $\mathbb H=\{z\in \mathbb C: z=x+iy, x>0\}.$ It is clear that the map $z\mapsto z^q$ is analytic on $\mathbb C\setminus (-\infty,0].$ Fix $a \in (0,\infty)$ and choose $r>0$ so that the open disc, centered at $a$ with radius $r>0,$ denoted by $D(a,r)$ is contained in $\mathbb H.$ Since $g$ is a non-constant complex-analytic function on $\mathbb H$ by open mapping theorem, $g(D(a,r))$ is also open. Since $1+a^q\in g(D(a,r)),$ we can choose an open ball with center at $1+a^q$ of radius $r^{\prime}>0$ such that $D(1+a^q,r^\prime)\subseteq g(D(a,r))\cap\mathbb H.$   Set $U=g^{-1}\left(D(1+a^q,r^{\prime})\right).$ Since $g$ is continuous, $U$ is an open set containing $a.$ Note that the map $H(z)=z^{\frac{1}{q}}$ is complex-analytic on $\mathbb H.$ Therefore, $h=H\circ g|_U$ is complex-analytic on $U.$ This implies that $f_q$ is real-analytic at $a.$ Since $a$ is an arbitrary point in $(0,\infty),$ $f_q$ is real-analytic on $(0,\infty).$ This completes the proof of the proposition.
\end{proof}

\begin{thm}\label{isometric}
Let $1<q<\infty$. If there exists a linear isometry of $S^{m}_{q}$ into $S^{n}_{\infty}$ for $n\geq m\geq 2$, then we must have $q=2$.
\end{thm}
\begin{proof}
	Suppose there exists a linear isometry of $S^{m}_{q}$ into $S^{n}_{\infty}$ for $n\geq m\geq 2$, then by Lemma \ref{reduct}, we may assume that there exist $A,B\in S_\infty^n$ such that $(A,B)$ has $(\mathbf{I}_{q,\infty}).$
	
	\noindent{\bf{Case-$\bf{1}$}}: Let $q>2$. Let $x$ be a unit eigenvector corresponding to eigenvalue $1$ or $-1$ of the diagonal self-adjoint matrix $A$, that is $Ax=x$ or $Ax=-x$ with $\|x\|=1$. Then for each $t\in\R$, $(\mathbf{I}_{q,\infty})$ implies that 
	\begin{align*}
		&\|(A+tB)x\|\leq \|(A+tB)\|_\infty=(1+|t|^q)^{1/q}\\
		\implies& \big[1+t(\la Ax,Bx\ra+\la Bx,Ax\ra )+t^2\|Bx\|^2\big]^q\leq 1+2|t|^q+|t|^{2q},
	\end{align*} 
	which forces that $Bx=0$. Let $P$ be the orthogonal projection on the eigenspaces corresponding to the eigenvalues $1$ and $-1$ of $A$. It is easy to check that $B=(I-P)B(I-P)$ and $\|(I-P)A(I-P)\|_\infty<1$. Therefore, for small $t$, we have
	\[\|A+tB\|_\infty=\max\big\{\|PAP\|_\infty,\|(I-P)(A+tB)(I-P)\|_\infty\big\}=1,\] which is a contradiction as $\|A+tB\|_\infty=(1+|t|^q)^{1/q}$ is true for all $t\in\R$.
	
	\noindent{\bf{Case-$\bf{2}$}}: Let $1<q<2$. In this case we show that $q$ must be an integer, in other word $q$ must be equal to $2$. Suppose not. By Theorem \ref{kato} we have
	\begin{equation}\label{A+tB op}
		\|A+tB\|_{\infty}=(1+|t|^q)^{\frac{1}{q}}=\max_{1\leq i\leq n}|\lambda_{i}(t)|.
	\end{equation}
	We will show that there do not exist real-analytic functions $\lambda_1,\dots,\lambda_n$ on a neighborhood of zero, say $(-\epsilon,\epsilon)$ with $\epsilon>0$ such that
	\begin{equation}\label{udiyed}
		(1+|t|^q)^{\frac{1}{q}}=\max_{1\leq i\leq n}|\lambda_{i}(t)|, ~\text{for all},~ t\in (-\epsilon,\epsilon).
	\end{equation}
	Note that in equation \eqref{udiyed} if $\{t\in(-\epsilon,\epsilon):|\lambda_r(t)|=|\lambda_s(t)|\}=(-\epsilon,\epsilon)$ for some $1\leq r\neq s\leq n$, then we can rewrite equation \eqref{udiyed} as
	\begin{equation}\label{eq 3.4}
		(1+|t|^q)^{\frac{1}{q}}=\max\{|\lambda_1(t)|,\dots,|\lambda_{r-1}(t)|,|\lambda_{r+1}(t)|,\dots,|\lambda_n(t)|\}~~{\rm{for ~~all}}~~t~~{\rm{with}} ~|t|<\epsilon. 
	\end{equation}
	Therefore, in view of \eqref{eq 3.4}, we can always assume that in equation \eqref{udiyed}, we also have that 
	\begin{equation}\label{eq 3.5}
		\{t\in(-\epsilon,\epsilon):|\lambda_i(t)|=|\lambda_j(t)|\}\neq(-\epsilon,\epsilon)\ \text{for all } 1\leq i\neq j\leq n.
	\end{equation}
	
	Let $1\leq i_0\leq n$ be the largest integer such that there exists an increasing sequence $k_1<k_2<\cdots<k_{i_0}$ with $1\leq k_j\leq n$ for all $1\leq j\leq i_0$ and the set \[\mathcal A_{k_1,\dots,k_{i_0}}:=\{t \in (-\epsilon,\epsilon):|\lambda_{k_r}(t)|\neq  |\lambda_{k_s}(t)|~\text{for all}, 1\leq r\neq s\leq i_0\}\] is non-empty. Clearly, by \eqref{eq 3.5} we must have $i_0\geq 1.$ 
	
	Let us choose a set $\mathcal A_{k_1,\dots,k_{i_0}}$ defined as above, which is non-empty. Note that the set $\mathcal A_{k_1,\dots,k_{i_0}}$ is open as $\mathcal A_{k_1,\dots,k_{i_0}}=\cap_{1\leq r\neq s\leq i_0}\{t \in (-\epsilon,\epsilon):|\lambda_{k_r}(t)|\neq  |\lambda_{k_s}(t)|\}$ and it is easy to see from continuity of the maps $\lambda_i$ that each of the sets $\{t \in (-\epsilon,\epsilon):|\lambda_{k_r}(t)|\neq  |\lambda_{k_s}(t)|\}$ is open for all $1\leq r\neq s\leq n$. Let us consider the following cases.
	
	\noindent{\bf{Sub-case-$\bf{1}$}}: Suppose $i_0=n.$ Then, we have that the set $\mathcal A_{1,\dots,n}$ is non-empty. Suppose $(a,b)\subseteq\mathcal A_{1,\dots,n}.$ Moreover, we choose $(a,b)$ so that either $(a,b)\subseteq(-\epsilon,0)$ or $(0,\epsilon).$ Let $t_0\in(a,b).$ Then we must have for some $1\leq l\leq n,$ \[|\lambda_l(t_0)|>\max_{1\leq i\neq l \leq n}|\lambda_i(t_0)|.\] By continuity, the same is true for some open interval $(c,d)\subseteq(a,b)$, i. e. we have for all $t\in(c,d)$
	\[|\lambda_l(t)|>\max_{1\leq i\neq l \leq n}|\lambda_i(t)|.\] Therefore, it follows from \eqref{udiyed} that 
	\begin{equation}\label{(aj,bj)}
		|\lambda_{l}(t)|=(1+|t|^q)^{\frac{1}{q}},~~t\in(c, d).
	\end{equation}
	Note that it follows from \ref{(aj,bj)} that on some smaller interval $(c^\prime,d^\prime)\subseteq (c,d)$ we must have $\lambda_{l}(t)=(1+|t|^q)^{\frac{1}{q}},\ t\in(c^\prime,d^\prime)$ or $\lambda_{l}(t)=-(1+|t|^q)^{\frac{1}{q}},\ t\in(c^\prime,d^\prime).$ Therefore, as the map $\lambda_l$ is analytic, and also from Proposition \ref{analytic prom} the map $t\mapsto (1+|t|^q)^{\frac{1}{q}}$ is analytic on $(c^\prime,d^\prime),$ by identity theorem we may assume that $(c^\prime,d^\prime)=(-\epsilon,0),~\text{or}~(0,\epsilon).$ Since $\lambda_l$ is analytic on $(-\epsilon,\epsilon)$ this means the map $t\mapsto (1+|t|^q)^{\frac{1}{q}}$ can be analytically extended on a neighborhood of zero. Let us denote $[q]$ to be the greatest integer less than or equal to $q$ and $\{q\}=q-[q].$ Since $q\notin\mathbb N$ we have that $0<\{q\}<1.$ We have from Lemma \ref{derivative of f} that if $t>0,$ then $f_q^{[q]+1}(t)=t^{\{q\}-1}(1+t^q)^{\frac{1}{q}-[q]-1}g_{[q]+1}(t^q).$ Therefore, $\lim\limits_{t\to 0^{+}}f_q^{[q]+1}(t)$ does not exist as $g_{[q]+1}$ is continuous at $0$ and $g_{[q]+1}(0)\neq 0.$ Similarly, $\lim\limits_{t\to 0^{-}}f_q^{[q]+1}(t)$ does not exist. This contradicts the fact that $f_q$ can be extended analytically on a neighborhood of $0$.

	\noindent{\bf{Sub-case-$\bf{2}$}}: Suppose $i_0<n.$  Suppose $(a,b)\subseteq\mathcal A_{k_1,\dots,k_{i_0}}.$ Moreover, we choose $(a,b)$ so that either $(a,b)\subseteq(-\epsilon,0)$ or $(0,\epsilon).$ Let $\{r_1,\dots,r_{j_0}\}=\{1,\dots,n\}\setminus\{k_1,\dots,k_{i_0}\}.$ Let us consider $\mathcal B_{r_1,k_l}:=\{t\in (a,b):|\lambda_{r_1}(t)|=|\lambda_{k_l}(t)|\}$ where $1\leq l\leq i_0.$ For a fixed but arbitrary $l,$ we observe that $\mathcal B_{r_1,k_l}=\{t\in(a,b):\lambda_{r_1}(t)=\lambda_{k_l}(t)\}\cup \{t\in(a,b):\lambda_{r_1}(t)=-\lambda_{k_l}(t)\}$. Note that the set $\{t\in(a,b):\lambda_{r_1}(t)=\lambda_{k_l}(t)\}$ is a zero set of an analytic function. Therefore, by identity theorem, it is either empty or discrete or $(a,b).$ However, if $\{t\in(a,b):\lambda_{r_1}(t)=\lambda_{k_l}(t)\}=(a,b),$ then again by identity theorem $\{t\in(-\epsilon,\epsilon):\lambda_{r_1}(t)=\lambda_{k_l}(t)\}=(-\epsilon,\epsilon),$ which is not possible. Therefore, we conclude that the set $\{t\in(a,b):\lambda_{r_1}(t)=\lambda_{k_l}(t)\}$ is either discrete or empty. The same is true for the set $\{t\in(a,b):\lambda_{r_1}(t)=-\lambda_{k_l}(t)\}$. Therefore, each of the set $\mathcal{B}_{r_1,k_l}$ is discrete or empty. Now consider the set $(a,b)\setminus\cup_{l=1}^{i_0}\mathcal B_{r_1,k_l}.$ Clearly, $(a,b)\setminus\cup_{l=1}^{i_0}\mathcal B_{r_1,k_l}$ is non-empty as each of the sets $\mathcal B_{r_1,k_l}$ is discrete or empty for all $1\leq l\leq i_0.$ Therefore, there exists $t_0\in(a,b)$ such that $|\lambda_{r_1}(t_0)|\neq |\lambda_{k_l}(t_0)|$ for all $1\leq l\leq i_0.$ This contradicts with the choice of $i_0.$

	This completes the proof of the theorem.
\end{proof}
\begin{rmrk}
Note that Case-1 of Theorem \ref{isometric} was already solved in \cite[Theorem 1.2]{Ra20}. However, our proof is much simpler and avoids use of Birkhoff-James orthogonality. 
\end{rmrk}
Observe that the methods for Theorem \ref{isometric} simply breaks down when we study isometric embeddability of $S_q^m$ into $S_1^n,$ $1<q<\infty$ and $2\leq m\leq n$ as the description of the $S_1^n$-norm prevents us to use the full strength of Kato-Rellich theorem. On the other hand, the methods of Theorem \ref{th2} do not work as the crucial Lemma \ref{l2} is not valid for $p=1$. To overcome these difficulties we need the following theorem which gives important geometric properties of $S_1^n.$

\begin{thm}\cite[Theorem 4.3]{DaGaJa84}\label{th}
	Suppose that $\mathcal{A}$ is a $C^*$-algebra. Then $\mathcal{A}^*$ is $2$-uniformly PL-convex with constant $c\geq \frac{1}{2}$, that is 
	\begin{align*}
	\|x\|^2+c\|y\|^2\leq \frac{1}{2\pi}\int_{0}^{2\pi}\|x+e^{it}y\|^2 \quad\forall x,y \in\mathcal{A^*}.
	\end{align*}
\end{thm}
\begin{thm}\label{thm 3.4}Let $1<q<\infty$. If there exists a linear isometry of $S_q^m$  into $S_1^n,~2\leq m\leq n,$ then we must have $q\in\{2,3\}.$
\end{thm}
\begin{proof}
	By Lemma \ref{reduct}, we may assume that there exist $A,B\in S_\infty^n$ such that $(A,B)$ has $(\mathbf{I}_{q,1}).$
	
	\noindent{\bf{Case-$\bf{1}$}}: Suppose $q\geq 4$. Then from Theorem \ref{th} along with $(\mathbf{I}_{q,1})$, it follows that $(1+c)^q\leq 4 $, which is impossible.
	
	\noindent{\bf{Case-$\bf{2}$}}: Let $1<q<4$. Then, by Theorem \ref{kato}, there exists $\epsilon>0$ such that we have the identity
	\[(1+|t|^q)^{\frac{1}{q}}=\sum_{i=1}^n|\lambda_i(t)|,\ \text{for all}\ \in(-\epsilon,\epsilon).\] We will show that the above identity forces $q$ to be an integer. Without loss of any generality, let us also assume that for all $1\leq i\leq n,$ $\lambda_i$ does not vanish everywhere. Note that for all $1\leq i\leq n,$ the set $\{t\in(-\epsilon,\epsilon):\lambda_i(t)=0\}$ is either empty or discrete. Therefore, $(-\epsilon,\epsilon)\setminus\cup_{i=1}^n\{t\in(-\epsilon,\epsilon):\lambda_i(t)=0\}$ is a non-empty open set. Therefore, there exists an interval $(a,b)$ which is contained in either $(-\infty,0)$ or $(0,\infty),$ such that we have $\lambda_i(t)\neq 0$ for all $1\leq i\leq n$ and $t\in(a,b).$ This implies that there exists a choice of $\epsilon_i\in\{+1,-1\}$ such that $\sum_{i=1}^n\epsilon_i\lambda_i(t)=\sum_{i=1}^n|\lambda_i(t)|$ for all $t\in (a,b).$ The required conclusion follows as arguing in the \textbf{Case-1} of the proof of Theorem \ref{isometric}. This completes the proof of the theorem.
\end{proof}
\section{Isometric embeddability for finite dimensional case:}\label{isofinite}
We prove the following lemma which might be of independent interest.
\begin{lma}\label{l2}
	Let $A$ and $B$ be two operators on $\ell_2^n$ with $(\mathbf{I}_{q,p})$.% $\|A\|_p=\|B\|_p=1$ and $(1+|t|^q)^{p/q}=\norm{A+tB}_p^p$ for all $t\in \mathbb{R}$.
	\begin{enumerate}[(i)]
		\item %If the map $\mathbb{R}\ni t\mapsto A+tB\in \mathcal{B}(\hil)$ is not invertible in any neighborhood of $0$, 
		Then for $p\geq 2$
		\begin{align*}
			\frac{d^2}{dt^2}\bigg|_{t=0}\norm{A+tB}_p^p> 0.
		\end{align*}
		\item If the map $\mathbb{R}\ni t\mapsto A+tB\in \mathcal{B}(\ell_2^n	)$ is invertible in a neighborhood of $0$, then for $p>1$
		\begin{align*}
			\frac{d^2}{dt^2}\bigg|_{t=0}\norm{A+tB}_p^p> 0.
		\end{align*}
	\end{enumerate}
\end{lma}
\begin{proof}
	$(i)$~%Suppose that the map $\mathbb{R}\ni t\mapsto A+tB\in \mathcal{B}(\hil)$  is not invertible in any neighborhood of $0$.
	Let $f_p$ be the \textquotedblleft{smooth}\textquotedblright  -function $|\cdot|^p$, i. e., $f_p$ is a smooth compactly supported function on $\mathbb{R}\setminus \{0\}$ such that 
	\begin{align}\label{a}
		f_p(x) &=|x|^p \quad \text{for}\quad x\in [-2,2].
		%\begin{cases*}
		%|x|^p \quad \text{for}\quad x\in [-2,2],\\
		%0 \quad \quad \text{for}\quad \mathbb{R}\setminus [-2,2].
		%\end{cases*}
	\end{align}
	Then, for $p\geq2$, it is easy to observe that $f_p\in C^2(\mathbb{R})$ and clearly
	\[\norm{A+tB}_p^p=\operatorname{Tr}(|A+tB|^p)=\operatorname{Tr}(f_p(A+tB)),~ t\in[-1,1].\] 
	Therefore, by using Theorem~\ref{th1} we conclude that  
	\begin{align}\label{eq4}
		\frac{1}{2!}\frac{d^2}{dt^2}\bigg|_{t=0}\operatorname{Tr}(f_p(A+tB))=\operatorname{Tr}\big(T^{A,A,A}_{f_p^{[2]}}(B,B)\big).
	\end{align}
	Let $\{d_i\}_{i=1}^{n}$ be the set of eigenvalues of $A$ and let $\{e_i\}_{i=1}^{n}$ be the corresponding set of eigenvectors of $A$ which forms an orthonormal basis for $\ell_2^n$. Let $P_i$ be the orthogonal projection of $\ell_2^n$ onto the space spanned by the vector $\{e_i\}$. Then the matrix representation of the operators $A$, $B$ and $P_i$  with respect to the orthonormal basis $\{e_i\}_{i=1}^{n}$ are as follows: \\  
	$A=\begin{bmatrix}
		d_1&0&\cdots&0\\
		0&d_2&\cdots&0\\
		\vdots&\vdots&\ddots&0\\
		0&0&\cdots&d_n
	\end{bmatrix}_{n\times n},$ 
	$B=\begin{bmatrix}
		b_{11}&b_{12}&\cdots&b_{1n}\\
		\bar{b}_{12}&b_{22}&\cdots&b_{2n}\\
		\vdots&\vdots&\ddots&0\\
		\bar{b}_{1n}&\bar{b}_{2n}&\cdots&{b}_{nn}
	\end{bmatrix}_{n\times n}$ and 
	$P_i=\begin{bsmallmatrix}
		0&\cdots&0&\cdots&0\\
		\vdots&\ddots&\vdots&&\vdots\\
		0&\cdots&\boxed{1_{ii}}&\cdots&0\\
		\vdots&&\vdots&\ddots&\vdots\\
		0&\cdots&0&\cdots&0
	\end{bsmallmatrix}_{n\times n}.$\\ 
	Now by using  (\ref{eq1}) and the cyclicity property of trace we get
	\begin{align}\label{eq5}
		\nonumber\Tr\big(T^{A,A,A}_{f_p^{[2]}}(B,B)\big)=&\sum_{i_0,i_1,i_2=1}^{n}f_p^{[2]}(d_{i_0},d_{i_1},d_{i_2})\Tr(P_{i_0}BP_{i_1}BP_{i_2})\\
		\nonumber=&\sum_{i_0,i_1,i_2=1}^{n}f_p^{[2]}(d_{i_0},d_{i_1},d_{i_2})\Tr(BP_{i_1}BP_{i_2}P_{i_0})\\
		\nonumber=&\sum_{i_0,i_1=1}^{n}f_p^{[2]}(d_{i_0},d_{i_1},d_{i_0})\Tr(BP_{i_1}BP_{i_0})\\
		\nonumber=&\sum_{i_0,i_1=1}^{n}f_p^{[2]}(d_{i_0},d_{i_1},d_{i_0})|b_{i_0i_1}|^2\\
		=&\sum_{l=1}^{n}|b_{ll}|^2f_p^{[2]}(d_l,d_l,d_l)+\sum_{l=1}^{n-1}\sum_{k=l+1}^{n}~|b_{lk}|^2~\big(f_p^{[2]}(d_l,d_k,d_k)+f_p^{[2]}(d_k,d_l,d_l)\big).
	\end{align}
	In above the last equality follows from the previous one, by collecting indices which are equal at one place and the rest in the other. On the other hand, from the definition of divided difference, it follows that if $d_l\neq 0$ then $f_p^{[2]}(d_l,d_l,d_l)>0 $ and if $d_l=d_k\neq 0$ then $f_p^{[2]}(d_l,d_k,d_k)+f_p^{[2]}(d_k,d_l,d_l)>0$. Moreover, for $d_l\neq d_k$, we have
	\begin{align}\label{e}
		\nonumber f_p^{[2]}(d_l,d_k,d_k)+f_p^{[2]}(d_k,d_l,d_l)=&~\frac{\partial}{\partial\lambda}\bigg|_{\lambda=d_k}\Bigg(\dfrac{|d_l|^p-|\lambda|^p}{d_l-\lambda}\Bigg)+\frac{\partial}{\partial\lambda}\bigg|_{\lambda=d_l}\Bigg(\dfrac{|d_k|^p-|\lambda|^p}{d_k-\lambda}\Bigg)\\
		=&~p\times\dfrac{|d_l|^p+|d_k|^p-d_ld_k(|d_l|^{p-2}+|d_k|^{p-2})}{(d_l-d_k)^2}.
		%>&~0.
	\end{align} It is sufficient to show that right hand side of (\ref{e}) is strictly greater than $0$ when $d_l,d_k>0$. To show that we rewrite (\ref{e}) as follows 
	\begin{align}\label{eq6}
		f_p^{[2]}(d_l,d_k,d_k)+f_p^{[2]}(d_k,d_l,d_l)=&~p\cdot d_l^{p-2}\cdot\dfrac{1-\left(\frac{d_k}{d_l}\right)^{p-1}}{1-\left(\frac{d_k}{d_l}\right)},
	\end{align} which is strictly positive for $d_l\neq d_k.$
	Now from Lemma~\ref{l} we conclude that $AB\neq 0$ and hence there exist some $l,k\in\{1,2,\ldots,n\}$ such that $d_l\neq 0$ and $b_{lk}\neq 0$. %Indeed, if we assume that there does not exist any $l,k\in\{1,2,\ldots,n\}$ such that either $d_l= 0$ or $b_{lk} = 0$, 
	If not, then we can find an orthonormal basis with respect to which the matrix representations of $A$ and $B$ are as follows: 
	\begin{align*}
		A=\begin{bmatrix}
			d_1&0&\cdots&0&\cdots&0\\
			0&d_2&\cdots&0&\cdots&0\\
			\vdots&\vdots&\ddots&\vdots&&\vdots\\
			0&0&\cdots&d_k&\cdots&0\\
			\vdots&\vdots&&\vdots&0&\vdots\\
			0&0&\cdots&0&\cdots&0
		\end{bmatrix}_{n\times n}~\text{and }
		B=\begin{bmatrix}
			0&0&\cdots&0&\cdots&0\\
			0&0&\cdots&0&\cdots&0\\
			\vdots&\vdots&\ddots&\vdots&&\vdots\\
			0&0&\cdots&0&\cdots&0\\
			\vdots&\vdots&&\vdots&\neq 0&\vdots\\
			0&0&\cdots&0&\cdots&b_{nn}
		\end{bmatrix}_{n\times n},
	\end{align*} 
	and hence $AB=BA=0$ which is a contradiction. Thus from the above observation along with equations \eqref{eq5} and \eqref{eq6} we conclude that
	\begin{align*}
		\frac{d^2}{dt^2}\bigg|_{t=0}\norm{A+tB}_p^p> 0.
	\end{align*}
	$(ii)$~%~Next we assume that the map $\mathbb{R}\ni t\mapsto A+tB\in \mathcal{B}(\hil)$ is invertible in a neighborhood of $0$.Then from Lemma \ref{l1} it follows
	Note that there exists $a>0$ such that $\sigma(A+tB)\cap[-a,a]$ is empty for $t$ in a small neighborhood of $0$. Now we define the compactly supported smooth function $f_p:\R\rightarrow \mathbb{C}$ such that $f_p(x)=|x|^p$ on $[-2,-a]\cup[a,2]$. Thus by applying Theorem~\ref{th1} corresponding to the function $f_p$ we get 
	\[\frac{d^2}{dt^2}\bigg|_{t=0}\norm{A+tB}_p^p=\sum_{l=1}^{n}|b_{ll}|^2f_p^{[2]}(d_l,d_l,d_l)+\sum_{l=1}^{n-1}\sum_{k=l+1}^{n}~|b_{lk}|^2~\big(f_p^{[2]}(d_l,d_k,d_k)+f_p^{[2]}(d_k,d_l,d_l)\big).
	\] 
	Note that in this case we have $d_l \neq 0$ for $1\leq l\leq n$. Therefore by similar argument as in $(i)$ we conclude that 
	\[\frac{d^2}{dt^2}\bigg|_{t=0}\norm{A+tB}_p^p>0.\]
	This completes the proof. 
\end{proof}
Now we are in a position to state and prove our main theorem in this section.
\begin{thm}\label{th2}
	Let $(q,p)\in\left(1,\infty\right]\times\left(1,\infty\right) $  and $2\leq m\leq n<\infty$. If there exists an isometric embedding $S_q^m$ into $S_p^n$ with $p\neq q$,
	then $q$ must be $2$.
\end{thm}

\begin{proof}
	%On contrary, we assume that there exists an isometric embedding from $S_q$ to $S_p$ for  $2\leq p<\infty$, $1<q<\infty$ with $p\neq q$, it then follows that there is an isometric embedding $f:l_q\mapsto S_p$. Let $\{e_i\}_{i=1}^{q}$ be the standard orthonormal basis of $l_q$. Let $f(1,0,\ldots,0)=A$ and $f(0,1,\ldots,0)=B$. Since $f$ is an isometry, then we have the following identity:
	%\begin{align}\label{eq7}
	%(1+|t|^q)^{p/q}=\norm{A+tB}_p^p,~t\in\R.
	%\end{align} 
	%Without loss of generality we assume $A$ and $B$ are both  $(n\times n)$ self-adjoint matrices {\color{red}(need to explain more!!).}
	Note that by Lemma \ref{reduct}, it is enough to prove that if $A\in S_p^n$ and $B\in S_p^n$ are such that $(A,B)$ has $(\mathbf{I}_{q,p})$, then $q=2.$
	By Theorem \ref{kato}, there exist real-analytic functions $\lambda_1,\lambda_2,\ldots,\lambda_n$ on a neighborhood of $0$ such that $\{\lambda_1(t),\dots,\lambda_n(t)\}$ is the complete set of eigenvalues of $A+tB$. Since $(A,B)$ has $(\mathbf{I}_{q,p})$, we have
	\begin{align}\label{eq8}
		\|(1,t)\|_q^p=\sum_{j=1}^{n}|\lambda_j(t)|^p.
	\end{align} 
	\noindent \textbf{Case-I:} Let $(q,p)\in (1,\infty)\times(1,\infty)$, and assume $\lambda_i(0)\neq 0$ for all $i$. In other words $A+tB$ is invertible in a neighborhood of $0$. Then the right hand side of \eqref{eq8} is a real-analytic function on a neighborhood of the origin and hence $t\mapsto \norm{A+tB}_p^p$ is real-analytic. Therefore the map $t\mapsto \norm{A+tB}_p^p$ has a power-series representation in a neighborhood of $0$ and moreover by Lemma~\ref{l2} we have $\frac{d^2}{dt^2}\bigg|_{t=0}\norm{A+tB}_p^p\neq 0$. Let 
	\[\norm{A+tB}_p^p=\sum_{n=0}^{\infty}a_nt^n,\quad \text{where}\quad a_0=1 \quad \text{and}\quad a_2\neq 0.
	\] 
	On the other hand since $A\perp_{BJ} B$, then from Theorem~\ref{l3} we conclude that $a_1=0$. Therefore from \eqref{eq8} we get
	\begin{align*}
		\nonumber&(1+|t|^q)^{p/q}=\norm{A+tB}_p^p\\
		\nonumber \implies ~&1+\frac{p}{q}|t|^q+\mathcal{O}(|t|^{2q})=\sum_{n=0}^{\infty}a_nt^n.
	\end{align*}
	Taking $t>0,$ the above equation becomes to
	\begin{align}\label{eq9}
		&\frac{p}{q}t^{q-2}+\mathcal{O}(t^{2q-2})=a_2+\sum_{n=3}^{\infty}a_nt^{n-2}.
	\end{align}
	Since the above identity (\ref{eq9}) holds identically in a neighborhood of 0 and  $a_2\neq0$, then by considering the limit as $t\rightarrow 0^+$ on both sides of \eqref{eq9} we conclude that $q=2$.
	
	\vspace{0.1in}
	\noindent \textbf{Case-II:} 
	Suppose  $\lambda_j(0)=0$ for some $j\in\{1,2,\ldots,n\}$, then if necessary by multiplying both side of $A$ and $B$ by permutation matrices, we can assume that $\lambda_i(0)=0$ for $1\leq i\leq k$, $\lambda_i(0)\neq 0$ for $k+1\leq i\leq n$, and $(q,p)\in (1,\infty)\times(1,\infty)$. Note that  $\norm{A}_p=1$ implies that $k\neq n$. We can assume that without loss of generality $\lambda_{i}$ are not identically zero in a neighborhood of $0$, because if $\lambda_{i}$ is identically zero for some $i$ then we can neglect this for calculating the $p$-norm of $A+tB$.  Since each $\lambda_i$ is a real-analytic function, then there exists an analytic function $\mu_i$ such that 
	\begin{align*}
		\lambda_i(t)=t^{m_i}\mu_i(t), \quad \text{where} \quad \mu_i(0)\neq 0 \quad \text{for some} \quad m_i\in\Nat \quad \text{and}\quad 1\leq i\leq k.
	\end{align*}
	Now consider $\psi(t)=\sum\limits_{i=k+1}^{n}(\text{sgn}\{\lambda_i(0)\}\lambda_i(t))^p-1$, then it is easy to observe that $\psi$ is an analytic function on a small neighborhood of $0$ with $\psi(0)=0$. Therefore from \eqref{eq8} we get
	\begin{align}\label{eq10}
		\nonumber1+\frac{p}{q}|t|^q+\mathcal{O}(|t|^{2q})=&\sum_{i=1}^{k}|t|^{m_ip}|\mu_i|^p(t)+1+\psi(t)\\
		\implies ~\frac{p}{q}|t|^q+\mathcal{O}(|t|^{2q})=&\sum_{i=1}^{k}|t|^{m_ip}|\mu_i|^p(t)+t\psi_0(t),
	\end{align} 
	where $\psi(t)=t\psi_0(t)$ and $\psi_0$ is an analytic function. Let $t>0$, then the above equation \eqref{eq10} yields that 
	\begin{align}\label{eq11}
		\frac{p}{q}t^{q-1}+\mathcal{O}(t^{2q-1})=&\sum_{i=1}^{k}t^{m_ip-1}|\mu_i|^p(t)+\psi_0(t).
	\end{align} 
	Now by taking limit as $t\to 0^+$ on both sides of \eqref{eq11} we conclude that $\psi_0(0)=0$ and hence there exists an analytic function $\psi_1$ such that $\psi_0(t)=t\psi_1(t)$. Thus from \eqref{eq11}
	we observe that
	\begin{align}\label{eq12}
		\frac{p}{q}t^{q-2}+\mathcal{O}(t^{2q-2})=&\sum_{i=1}^{k}t^{m_ip-2}|\mu_i|^p(t)+\psi_1(t).
	\end{align}
	
	\noindent \textbf{Subcase-I:} Let $p\geq 2$ and note that in this case $\psi_1(0)=\frac{1}{2}\frac{d^2}{dt^2}\bigg|_{t=0}\norm{A+tB}_p^p> 0.$ Therefore by taking limit as $t\to 0^+$ on both sides of \eqref{eq12} we conclude
	\begin{align*}
		\lim_{t\to 0^+}\frac{p}{q}~t^{q-2}\neq 0,
	\end{align*} 
	which is possible only when $q=2$.
	
	\vspace{0.1in}	
	
	%Then 
	%\begin{align}\label{eq}
	%\norm{A+tB}_2^2=1+2t\Tr(AB)+t^2,~t\in\R.
	%\end{align}
	%Note that since $A\perp_{BJ}^p B$, then by Lemma~\ref{l3}, we have $\Tr(AB)=0$. Thus $(\mathbf{I}_{q,p})$ and \eqref{eq} together implies  that for $t>0$ 
	%\begin{align*}
	%\frac{p}{q}t^{q-2}+\mathcal{O}(t^{2q-2})=1,
	%\end{align*}
	%which by applying limit as $t\to 0^+$ on both sides yields that $q=2$.	
	
	%\textbf{Subcase-II} Let $1<p<2$ and $m_i\geq2,1\leq i\leq k$. Then by similar argument as above we have $q=2$.\vspace*{.2in}\\
	\vspace{0.1in}
	\noindent\textbf{Subcase-II:} Let $1<p<2$.\\
	If for some $l\leq k$, $m_1=\cdots=m_l=1$ and $m_{l+1},\ldots,m_k\geq 2$. Then from \eqref{eq12} we have 
	\begin{align}
		\nonumber\frac{p}{q}t^{q-2}+\mathcal{O}(t^{2q-2})=&\sum_{i=1}^{l}t^{p-2}|\mu_i|^p(t)+\sum_{i=l+1}^{k}t^{m_ip-2}|\mu_i|^p(t)+\psi_1(t)\\
		\implies ~\frac{p}{q}t^{q-2}+\mathcal{O}(t^{2q-2})=&\sum_{i=1}^{l}t^{p-2}|\mu_i|^p(0)+\mathcal{O}(t^{p-1})+\sum_{i=l+1}^{k}t^{m_ip-2}|\mu_i|^p(t)+\psi_1(t).
	\end{align} Now taking limit $t\to 0^+$ on both side of the above equation we have 
	\begin{align*}
		\lim_{t\to 0^+}\big\{\frac{p}{q}t^{q-2}-\sum_{i=1}^{l}t^{p-2}|\mu_i|^p(0)\big\}<\infty \text{, which is impossible as }p\neq q.
	\end{align*}
	Thus we conclude $m_i\geq 2,~1\leq i \leq k$. Then the map $t\mapsto \norm{(A+tB)}_p^p~$ is twice differentiable at $t=0$ and from (\ref{eq12}), $\frac{d^2}{dt^2}\bigg|_{t=0}\norm{A+tB}_p^p=2 \psi_1(0)$. Note that $\lambda_i'(0)=\lambda_i''(0)=0$ for $1\leq i \leq k$. We claim that $\psi_1(0)\neq 0$. To establish this, first we choose $\epsilon>0$ such that
	\begin{itemize}
		\item[$\bullet$] $A_x=A+xI$ is invertible for each $x\in (0,\epsilon)$,
		\item[$\bullet$] $(\lambda_i(0)+x)$ have same sign for each $x\in (0,\epsilon)$.
	\end{itemize}
	Therefore 
	\begin{align}\label{a1}
		\norm{A_x+tB}_p^p= \sum_{i=1}^{n}|\lambda_i(t)+x|^p, \quad x\in(0,\epsilon).
	\end{align}
	Twice differentiating the above equation (\ref{a1})  with respect to $t$ in a neighborhood of $0$, we have
	\begin{align}\label{a2}
		\nonumber\frac{d^2}{dt^2}\norm{A_x+tB}_p^p=&  \sum_{i=1}^{n}\Big\{p|(\lambda_i(t)+x)|^{p-1}\sgn(\lambda_i(t)+x)\lambda_i''(t)\\
		&\quad+p(p-1)|(\lambda_i(t)+x)|^{p-2}(\lambda_i'(t))^2~\Big\}.
	\end{align} Now put $t=0$ in \eqref{a2} and then taking $x\to 0^+$ in both sides of \eqref{a2} we have,
	\begin{align}\label{a3}
		\nonumber\lim_{x\to 0^+}\at{\frac{d^2}{dt^2}}{t=0}\norm{A_x+tB}_p^p=&\sum_{i=k+1}^{n}\Big\{ p|\lambda_i(0)|^{p-1}\sgn(\lambda_i(0))\lambda_i''(0) + p(p-1)|\lambda_i(0)|^{p-2}(\lambda_i'(0))^2~\Big\}\\
		=&\at{\frac{d^2}{dt^2}}{t=0}\norm{A+tB}_p^p<\infty.
	\end{align}
	Let $d_i=\lambda_i(0),~1\leq i\leq n$, then  using \eqref{eq5}, we have
	\begin{align}\label{a4}
		\frac{1}{2}\frac{d^2}{dt^2}\norm{A_x+tB}_p^p	=&\sum_{i=1}^{n}|b_{ii}|^2f_p^{[2]}(d_i^x,d_i^x,d_i^x)+\sum_{i=1}^{n-1}\sum_{j=i+1}^{n}~|b_{ij}|^2~\big(f_p^{[2]}(d_i^x,d_j^x,d_j^x)+f_p^{[2]}(d_j^x,d_i^x,d_i^x)\big),
	\end{align}
	where $d_i^x=d_i+x,~1\leq i \leq n$ are the eigenvalues of $A_x$. In (\ref{a3}), we have seen that $\lim\limits_{x\to 0^+}\at{\frac{d^2}{dt^2}}{t=0}\norm{A_x+tB}_p^p$ is finite, which forces to conclude that $b_{ij}=0,~1\leq i\leq k, i\leq j\leq k$. Therefore (\ref{a4}) becomes to
	\begin{align}\label{a5}
		\nonumber &\frac{1}{2}\frac{d^2}{dt^2}\norm{A_x+tB}_p^p\\	\nonumber=&\sum_{i=k+1}^{n}|b_{ii}|^2f_p^{[2]}(d_i^x,d_i^x,d_i^x)+\sum_{i=1}^{k}\sum_{j=k+1}^{n}~|b_{ij}|^2~\big(f_p^{[2]}(d_i^x,d_j^x,d_j^x)+f_p^{[2]}(d_j^x,d_i^x,d_i^x)\big)\\
		&\hspace*{1in}+\sum_{i=k+1}^{n-1}\sum_{j=i+1}^{n}~|b_{ij}|^2~\big(f_p^{[2]}(d_i^x,d_j^x,d_j^x)+f_p^{[2]}(d_j^x,d_i^x,d_i^x)\big).
	\end{align} Now taking limit $x\to 0^+$ on both sides of (\ref{a5}), we have 
	\begin{align}\label{a6}
		\nonumber&\lim_{x\to 0^+}\frac{1}{2}\frac{d^2}{dt^2}\norm{A_x+tB}_p^p\\	\nonumber=&\sum_{i=k+1}^{n}|b_{ii}|^2f_p^{[2]}(d_i,d_i,d_i)+\sum_{i=1}^{k}\sum_{j=k+1}^{n}~|b_{ij}|^2~\big(f_p^{[2]}(d_i,d_j,d_j)+f_p^{[2]}(d_j,d_i,d_i)\big)\\
		&\hspace*{1in}+\sum_{i=k+1}^{n-1}\sum_{j=i+1}^{n}~|b_{ij}|^2~\big(f_p^{[2]}(d_i,d_j,d_j)+f_p^{[2]}(d_j,d_i,d_i)\big).
	\end{align}
	Therefore by similar kind of argument as done in the proof of lemma \ref{l2}, from (\ref{a6}), we have \[\lim_{x\to 0^+}\frac{d^2}{dt^2}\norm{A_x+tB}_p^p=\at{\frac{d^2}{dt^2}}{t=0}\norm{A+tB}_p^p=2\psi_1(0)\neq 0.\]
	
	Now by limiting argument, from (\ref{eq12}), we have $q=2$.

	%\end{proof}
	%\begin{thm}\label{sideth}
	%Let $1< p <\infty$. Then $S_\infty$ is not isometrically embedded into $S_p$.
	%\end{thm}
	%\begin{proof}
	\vspace{0.1in}
	\noindent \textbf{Case-III:} Let $(q,p)\in\{\infty\}\times(1,\infty)$.\\
	%On contrary, assume there is an isometric embedding from $S_\infty^m$ to $S_p^n$. Then there exist two self-adjoint matrix $A$ and $B$ such that 
	%\begin{align}\label{a7}
	%\max~\{1,|t|\}^{p}=\norm{A+tB}_p^p,\quad t\in \R.
	%\end{align}
	%By  \ref{th2} theorem, there exist real valued function $\lambda_1,\ldots,\lambda_n:\R\to\R$ such that 
	%\begin{itemize}
	%	\item for $t\in\R$, $\{\lambda_i(t)\}_{i=1}^{n}$  gives the full set of eigenvalue corresponds to the self-adjoint $A+tB$.
	%	\item $\lambda_i, (1\leq i \leq n)$ are analytic in a neighborhood of $1$.
	%\end{itemize}
	%Therefore the identity (\ref{a7}) can be rewritten as follows:
	%\begin{align}\label{a8}
	%\max~\{1,|t|\}^{p}=\sum_{i=1}^{n}|\lambda_i(t)|^p,\quad t\in \R.
	%\end{align}
	%Note that for $1<p<\infty$, the right hand side of (\ref{a8}) is differentiable at $t=1$, but the function $t\mapsto \max~\{1,|t|\}^{p}$ is not differentiable at $t=1$, which ensure that for $1<p<\infty$, $S_\infty$ is not isometrically embedded into $S_p$.
	This case will be taken care of in the next section as the proof is the same for both finite and infinite-dimensional cases.
\end{proof}
\begin{rmrk}
	In the proof of Theorem \ref{th2}, the methods in \textbf{Subcase-II} could be approached alternatively by an inequality proved in \cite{RiXu16}. Note that if $A$ is invertible, and $1<p<2$ it has been proved in the proof of \cite[Lemma 5]{RiXu16} (also see \cite{BaCaLi94}) that \[\frac{1}{p}\|A\|_p^{2-p}\at{\frac{d^2}{dt^2}}{t=0}\norm{A+tB}_p^p\geq (p-1)\|B\|_p^2.\] Therefore, by using Theorem \ref{kato} and using limiting argument as in \eqref{a6}, we have that
	\[\at{\frac{d^2}{dt^2}}{t=0}\norm{A+tB}_p^p\geq\frac{p(p-1)\|B\|_p^2}{\|A\|_p^{2-p}}=p(p-1)\] which is clearly non-zero.
\end{rmrk}
\section{Isometric embeddability of $S_q$ into $ S_p$}\label{isoinfinite}
%In this section, $\hil$ is an infinite separable Hilbert space.
\noindent \subsection{Multiple operator integrals on Schatten-p classes, approach via separation of variables:}
In this section we recall the definition of a multiple operator integral due to \cite{AzCaDoSu09,Pe06}. 
Let $\mathfrak{U}^n$ be the class of functions
$\phi:\R^{n+1}\to \mathbb{C}$ admitting the following representation 
\begin{align}\label{c}
	\phi(\lambda_0,\ldots,\lambda_n)=\int_\Omega b_0(\lambda_0,\omega)\cdots b_n(\lambda_n,\omega)d\nu(\omega),
\end{align}
for some finite measure space $(\Omega,\nu)$
and bounded Borel functions $b_j:\R\times\Omega\to\mathbb{C}$ for $0\leq j\leq n$ satisfying $\int_{\Omega}\prod_{j=0}^{n}\norm{b_j(\cdot,\omega)}_\infty~d|\nu|(\omega)<\infty$. The class $\mathfrak{U}^n$
has the norm
\[\norm{\phi}_{\mathfrak{U}^n}=\inf\int_\Omega\norm{b_0(\cdot,\omega)}_\infty\cdots\norm{b_n(\cdot,\omega)}_\infty d|\nu|(\omega), \]
where the infimum is taken over all possible representations \ref{c}. Now we are interested in the subclass $\mathfrak{C}^n\subset \mathfrak{U}^n$ of functions $\phi:\R^{n+1}\to \mathbb{C}$ admitting the representation \eqref{c} with bounded continuous functions
$b_j(\cdot,\omega):\mathbb{R}\mapsto \mathbb{C},$
for which there is a growing sequence $\{\Omega_k\}_{k\geq 1}$, with $\Omega_k\subseteq\Omega$ and $\cup_{k\geq 1}\Omega_k=\Omega$, such that the families $\{b_j(\cdot,\omega)\}_{\omega\in\Omega_k},~0\leq j\leq n$, are uniformly bounded and uniformly equicontinuous and given $\epsilon>0$, there exists $k_\epsilon\in\Nat$, for which 
\[\int_{\Omega\setminus\Omega_{k_{\epsilon}}}\prod_{j=0}^{n}\norm{b_j(\cdot,\omega)}_\infty~d|\nu|(\omega)<\epsilon.
\]
\begin{dfn}\label{dfn1}
	Let $H_0,H_1,\ldots,H_n$ be (possibly unbounded) self-adjoint operators in $\ell_2$ and $V_1,V_2,\ldots,V_n$ bounded self-adjoint operators in $\ell_2$. For $\phi\in \mathfrak{U}^n$, the multiple operator integral $T_\phi^{H_0,\ldots,H_n}(V_1,\ldots,V_n)$ is an operator defined for every $y\in \ell_2$ as the Bochner integral 
	\[
	T_\phi^{H_0,\ldots,H_n}(V_1,\ldots,V_n)y:=
	\int_\Omega b_0(H_0,\omega)V_1b_1(H_1,\omega)\cdots V_nb_n(H_n,\omega)y~d\nu(\omega).
	\]
\end{dfn}
\noindent It was proved in \cite{AzCaDoSu09,Pe06} that the above definition is independent of the choice of the representation \eqref{c}. The following theorem discusses the estimate for the above defined multiple operator integral.
\begin{thm}
	(See [\cite{PoSkSuOP13}, Proposition 2.2] and [\cite{AzCaDoSu09}, Remark 4.2])
	Let $1\leq p_j\leq \infty$, with $1\leq j\leq n$, be such that $0\leq \frac{1}{p}:=\frac{1}{p_1}+\cdots +\frac{1}{p_n}\leq 1$. Let $H_0$ be a (possibly unbounded) self-adjoint operator and $V_j=V_j^*\in S_{p_j}$, for $j\in\{1,\ldots,n\}$. Then for every $\phi\in\mathfrak{U}^n$, 
	\begin{equation*}
		\left\|T_\phi^{H_0,\ldots,H_0}(V_1,\ldots,V_n)\right\|_p\leq \|\phi\|_{\mathfrak{U}^n}\|V_1\|_{p_1}\cdots \|V_n\|_{p_n}. 
	\end{equation*}
\end{thm}

\noindent \subsection {Multiple operator integral on Schatten-p classes, approach via without separation of variables :}
In this section, we consider one more construction of a multiple operator integral which does not require a tensor product decomposition of a function $\phi$ as in \eqref{c}. In this construction \cite{PoSkSu13,SkTo19}, multiple operator integrals are represented as limits of Riemann sums with admissible partitions.

Let $E^j$ denote the spectral measure of $H_j$ for $0\leq j\leq n$. Now we set $E^j_{l,m}=E^j\Bigg(\Big[\frac{l}{m},\frac{l+1}{m}\Big)\Bigg)$, for every $m\in \mathbb{N}$ and $l\in \mathbb{Z}$. 
Let $\phi:\mathbb{R}^n\mapsto \mathbb{C}$ be a bounded continuous function. In case of convergence, denote
\begin{align*}
	&\hat{T}^{H_0,\ldots,H_n}_\phi(V_1,V_2,\ldots,V_n)\\
	:=&~\text{SOT}-\lim_{m\to\infty}~\norm{\cdot}_p-\lim_{N\to\infty}\sum_{|l_0|,|l_1|,\ldots,|l_n|\leq N}\phi\Big(\frac{l_0}{m},\frac{l_1}{m},\ldots,\frac{l_n}{m}\Big)E^0_{l_0,m}V_1E^1_{l_1,m}V_2E^2_{l_2,m}\cdots V_n E^n_{l_n,m},
\end{align*}
where the first limit gives bounded polylinear operators and the second one is evaluated in the strong operator topology on the tuples $(V_1,\ldots,V_n)\in S_{p_1}\times\cdots\times S_{p_n}$, where 
$\frac{1}{p}=\frac{1}{p_1}+\frac{1}{p_2}+\cdots+\frac{1}{p_n}$. It is important to note that the values of the function $\phi$ outside the set $(\alpha,\beta)^n$, where the interval $(\alpha,\beta)$ contains the spectra of $H_0,\ldots,H_n$, do not affect the values of $\hat{T}^{H_0,\ldots,H_n}$. The following theorem shows that the multiple operator integral $\hat{T}^{H_0 \ldots,H_n}$ coincides with the
multilinear transformation given by Definition~\ref{dfn1}.
\begin{thm}\label{equalitythm}
	(See [\cite{PoSkSuOP13}, Proposition 2.7] and [\cite{PoSkSu13}, Lemma 3.5])
	Let $1\leq p_j\leq \infty$, with $1\leq j\leq n$, be such that $0\leq \frac{1}{p}:=\frac{1}{p_1}+\cdots +\frac{1}{p_n}\leq 1$. Let $H_0$ be a (possibly unbounded) self-adjoint operator and $V_j=V_j^*\in S_{p_j}$, for $j\in\{1,\ldots,n\}$. Then for every $\phi\in \mathfrak{C}^n$, 
	\begin{align*}
		&\hat{T}^{H_0,H_0,\ldots,H_0}_\phi :S_{p_1}\times S_{p_2}\times\cdots\times S_{p_n}\mapsto S_{p}\\
		& \hspace{1in} (V_1,V_2,\ldots,V_n)\mapsto \hat{T}^{H_0,H_0,\ldots,H_0}_\phi(V_1,V_2,\ldots,V_n)
	\end{align*}
	is a bounded polylinear operator mapping. Moreover, $\hat{T}^{H_0,H_0,\ldots,H_0}_\phi(V_1,V_2,\ldots,V_n)$ coincides with 
	$T^{H_0,H_0,\ldots,H_0}_\phi(V_1,V_2,\ldots,V_n)$
	given by Definition~\ref{dfn1}.
\end{thm}
\noindent Going further, we need the following few results towards obtaining our main result in this section. Now we state those essential results without proof.   
%\begin{lma}(\cite[Theorem 18]{PoSu14})\label{l4}
%Let $f_p$ be compactly supported $C^\infty$ function on $\R\setminus\{0\}$ such that $f_p(x)=|x|^p$ on $[-2,2]$. Then for every natural number $k<p$, $(f')^{[k-1]}\in\mathfrak{C}_{k-1}$.
%\end{lma}
\begin{thm}(See \cite[Theorem 18]{PoSu14})\label{th4}
	Let $H$ be a bounded self-adjoint operator in  $S_p$ and let $f_p$ be the function given in \eqref{a}. Then the operator $T_{(f'_p)^{[k]}}^{H,\ldots,H}$ maps $\underbrace{S_p\times\cdots\times S_p}_{(k-1)-times}$ into $S_{p'}$, and $T_{(f'_p)^{[k]}}^{H,\ldots,H}:\underbrace{S_p\times\cdots\times S_p}_{(k-1)-times}\mapsto S_{p'}$ is continuous for every integral $2\leq k<p$. Moreover, there exists a constant $C>0$ such that \[\norm{T_{(f'_p)^{[k]}}^{H,\ldots,H}(V_1,\ldots,V_k)}_{p'}\leq C	~\norm{H}_p^{p-k}\norm{V_1}_p\cdot \ldots \cdot \norm{V_k}_p,\] where $\frac{1}{p}+\frac{1}{p'}=1$ and $V_j\in S_{p}$.
\end{thm} 
\noindent The following remark essentially follows from the proof of the above theorem (See \cite[Theorem 18]{PoSu14}). 
\begin{rmrk}\label{l4}
	Let $f_p$ be the function given in \eqref{a}, then for $p>2$, $(f'_p)^{[k-1]}\in\mathfrak{C}^{k-1}$, where $k\in\Nat$ and $k<p$.
\end{rmrk}
\begin{lma}(See \cite[Lemma 5.8.4]{SkTo19})\label{lm}
	For every $1<p<\infty$, the function $\norm{\cdot}_{L^p}^p$ is Fr\'{e}chet differentiable at $H\in L_p(\mathcal M)$ and the respective Fr\'{e}chet derivative is given by 
	\begin{align}
		D(\norm{H}_{L^p}^p)(V)=p\cdot\tau(V\cdot|H|^{p-1}\sgn(H)).
	\end{align}
\end{lma}
\begin{thm}(See \cite[Theorem 5.8.5]{SkTo19}, \cite[Theorem 16]{PoSu14})\label{th5}
	Let $m\geq2$ and $p\in\left(m,m+1\right]$. For self-adjoint operators $H,V_1,\ldots,V_m\in S_p$ with $\norm{H}_p\leq 1$, define the map $\delta_{k,p}^H:\underbrace{S_p\times\cdots\times S_p}_{k-times}\mapsto \mathbb{C}$ by 
	\begin{align*}
		\delta_{k,p}^H(V_1,\ldots,V_k):=\begin{cases*}
			p\Tr(V_1|H|^{p-1}\sgn(H)),&k=1\\
			\frac{1}{k!}\sum\limits_{\pi\in\operatorname{Sym}_k}\Tr(V_{\pi(1)}\cdot T^{H,\ldots,H}_{(f_p')^{[k-1]}} (V_{\pi(2)},\ldots,V_{\pi(k)})),&k=2,3,\ldots,m,
		\end{cases*}
	\end{align*} 
	where $\operatorname{Sym}_k$ is the group of all permutations of the set $\{1,\ldots,k\}$ and the function $f_p$ given as in \eqref{a}. Then
	$\delta_{k,p}^H$ is a symmetric multilinear bounded functional for every $k=1,\ldots,m$ and furthermore 
	\begin{align}\label{b}
		\norm{H+V}_p^p=\norm{H}_p^p+\sum_{k=1}^{m}\frac{1}{k}\delta_{k,p}^H(\underbrace{V,\ldots,V)}_{k-times}+\mathcal{O}(\norm{V}_p^p).
	\end{align}
\end{thm}
%\begin{crlre}\label{crl} From the above equation \eqref{b} it follows that
%\begin{align}\label{eqderivative}
%	\frac{d}{dt}\norm{H+tV}_p^p=p\Tr(V|H_t|^{p-1}\sgn(H_t)).
%\end{align}
%\end{crlre}
\begin{lma}\label{l5}
	Let $2\leq p<\infty$, $1<q<\infty$ with $p\neq q$. Let $A,B$ be two bounded operators in $S_p$ having $(\mathbf{I}_{q,p})$. Then \[\Tr(B~T^{A,A}_{(f'_p)^{[1]}}(B))\neq 0,\] where $f_p$ is the function defined in \eqref{a}.
\end{lma}
\begin{proof}
	Let $\{d_i\}_{i=1}^\infty$ be the set of eigenvalues of $A$ and let $\{e_i\}_{i=1}^{\infty}$ be the corresponding set of orthonormal basis of $\ell_2$ consisting the eigenvectors of $A$. Then the canonical decomposition of $A$ is given by $A=\sum\limits_{i=1}^{\infty}d_i\la \cdot,e_i\ra e_i$. Let $P_n$ be the orthogonal projection of $\ell_2$ onto $\text{Span}~\{e_i\}_{i=1}^{n}$ for every $n\in \mathbb{N}$. Then it is easy to observe that $P_n$ converges to the identity operator in strong operator topology. Now we denote $A_n=P_nAP_n$ and $B_n=P_nBP_n$. Then it is straight forward to show that $\{A_n\}$ and $\{B_n\}$ converges to $A$ and $B$ respectively in Schatten- $p$ norm. Moreover from Lemma \ref{l} we know that $AB\neq 0$. Thus there exists $n_0\in\Nat$ such that $\norm{A_{n_0}}_p,\norm{B_{n_0}}_p\geq1/2$ and $A_{n_0}B_{n_0}\neq 0$. Next we show that $\Tr(B~T^{A,A}_{(f'_p)^{[1]}}(B))$ can be approximated by $\Tr(B_n~T^{A,A}_{(f'_p)^{[1]}}(B_n))$. Indeed,
	\begin{align*}
		&|\Tr(B~T^{A,A}_{(f'_p)^{[1]}}(B))-\Tr(B_n~T^{A,A}_{(f'_p)^{[1]}}(B_n))|\\
		\leq&~|\Tr(B~T^{A,A}_{(f'_p)^{[1]}}(B-B_n))|+|\Tr((B-B_n)~T^{A,A}_{(f'_p)^{[1]}}(B_n))|.
	\end{align*}
	On the other hand by using Theorem \ref{th4} from the above equation we conclude
	\begin{align*}
		&|\Tr(B~T^{A,A}_{(f'_p)^{[1]}}(B))-\Tr(B_n~T^{A,A}_{(f'_p)^{[1]}}(B_n))|\\
		\leq &~\text{const}~ \big(\norm{B}_p \norm{A}_p^{p-1}\norm{B-B_n}_p+ \norm{B-B_n}_p\norm{A}_p^{p-1}\norm{B_n}_p\big)\to 0 \text{ as } n\to\infty.
	\end{align*} 
	Now we consider $A_n$ and $B_n$ are the operators on the finite dimensional Hilbert space $P_n\ell_2$. Suppose the matrix representations of $A_n$ and $B_n$ with respect to the orthonormal basis $\{e_i\}_{i=1}^{n}$ are as follows:
	$A_n=\begin{bmatrix}
		d_1&0&\cdots&0\\
		0&d_2&\cdots&0\\
		\vdots&\vdots&\ddots&0\\
		0&0&\cdots&d_n
	\end{bmatrix}_{n\times n}$ and 
	$B_n=\begin{bmatrix}
		b_{11}&b_{12}&\cdots&b_{1n}\\
		\bar{b_{12}}&b_{22}&\cdots&b_{2n}\\
		\vdots&\vdots&\ddots&0\\
		\bar{b}_{1n}&\bar{b}_{2n}&\cdots&{b}_{nn}
	\end{bmatrix}_{n\times n}$.\\
	Now by Remark \ref{l4} we have $(f'_{p})^{[1]}\in\mathfrak{C}_1$ and hence by Theorem  \ref{equalitythm} we conclude 
	\begin{equation}\label{equality2}
		T_{(f'_{p})^{[1]}}=\hat{T}_{(f'_{p})^{[1]}}.\end{equation}
	Therefore using equations \eqref{equality2} and \eqref{eq1} we get
	\begin{align*}
		\Tr(B_n~T^{A,A}_{(f'_p)^{[1]}}(B_n))=\Tr(B_n~T^{A_n,A_n}_{(f'_p)^{[1]}}(B_n))
		=\Tr(B_n~\hat{T}^{A_n,A_n}_{(f'_p)^{[1]}}(B_n))
		=\frac{1}{2}\Tr(\hat{T}^{A_n,A_n,A_n}_{(f_p)^{[2]}}(B_n,B_n)).
	\end{align*}
	Thus the identity in \eqref{eq5} yields
	\begin{align*}
		\Tr(B_n~T^{A,A}_{(f'_p)^{[1]}}(B_n))=&\sum_{l=1}^{n}|b_{ll}|^2f_p^{[2]}(d_l,d_l,d_l)+\sum_{l=1}^{n-1}\sum_{k=l+1}^{n}|b_{lk}|^2\big(f_p^{[2]}(d_l,d_k,d_k)+f_p^{[2]}(d_k,d_l,d_l)\big)
	\end{align*} 
	which is an increasing sequence and by similar argument as in the finite dimensional case we conclude 	$\Tr(B_{n_0}~T^{A_{n_0},A_{n_0}}_{(f'_p)^{[1]}}(B_{n_0}))>0$ and hence 
	\begin{align*}
		\Tr(B~T^{A,A}_{(f'_p)^{[1]}}(B))=\lim_{n\to\infty}\Tr(B_n~T^{A,A}_{(f'_p)^{[1]}}(B_n))\neq 0.
	\end{align*}
	This completes the proof.
\end{proof}

\noindent Now we are in a position to state and prove our main result in this section, which establishes infinite-dimensional analogs of the results in Section \ref{isofinite}.
\begin{thm}\label{th6}
	Let $(q,p)\in \left(1,\infty\right)\times\left[2,\infty \right)\cup[4,\infty)\times\{1\} \cup\{\infty\}\times\left( 1,\infty\right)\cup[2,\infty)\times\{\infty\}$ be such that $p\neq q$. If there exists an isometric embedding $S_q$ into $S_p$, then $q=2$.
\end{thm}
\begin{proof}
	Suppose there exists an isometric embedding from $S_q$ to $S_p$. Therefore by Lemma \ref{reduct} there exist two bounded self-adjoint operators $A,B\in S_p$ satisfying $(\mathbf{I}_{q,p})$.
	
	\vspace{.1in}
	\noindent\textbf{Case-I:} Suppose $1<q<\infty$ and $2\leq p<\infty$. For $p=2$, from $(\mathbf{I}_{q,p})$ we have 
	\begin{align*}
		1+\frac{2}{q} |t|^q+\mathcal{O}(|t|^{2q})= 1+2t\Tr(AB)+t^2.
	\end{align*}
	On the other hand by Lemma \ref{l3} we have $\Tr(AB)=0$. Thus by limiting argument as earlier, we can conclude that $q=2$, which is a contradiction.\\
	%\begin{align*}
	%\frac{2}{q} |t|^{q-2}+\mathcal{O}(|t|^{2q-2})=\norm{B}_2^2,
	%\end{align*} 
	%which by considering the limit as $t\rightarrow 0$ on both sides  of the above equation implies that $q=2$. \vspace*{.5cm}\\
	%Then as in \textbf{Case-I} above along with equations \eqref{eq7}, \eqref{b} and \eqref{eqderivative} we conclude that for%
	Now we assume  $p>2$. Then for $t$ belonging to a small neighborhood of $0$, from $(\mathbf{I}_{q,p})$ and equation \eqref{b}, we have
	\begin{align*}
		1+\frac{2}{q} |t|^q+\mathcal{O}(|t|^{2q})=1+t\Tr(B\cdot|A|^{p-1}sgn(A))+\frac{t^2}{2} \Tr(B~T^{A,A}_{(f'_p)^{[1]}}(B))+\mathcal{O}(|t|^p).
	\end{align*} Therefore by applying Lemma \ref{l3} and Lemma \ref{l5} and performing limiting argument as in Case-II of Theorem~\ref{th2} we conclude that $q=2$.
	
	\vspace{.1in}
	\noindent
	\textbf{Case-II:} For the case when $(q,p)\in[4,\infty)\times\{1\}$ one can prove it by the same argument as in Case 1. of Theorem \ref{thm 3.4}.
	
	\vspace{.1in}
	\noindent
	\textbf{Case-III:} Let $q=\infty$ and $1<p<\infty$. Then $(\mathbf{I}_{q,p})$ gives
	\begin{align}\label{inf}
		\max\{1,|t|^p\}=\norm{A+tB}_p^p,\quad t\in\R.
	\end{align}
	Observe that the right hand side of the above equation \eqref{inf} is differentiable at $t=1$ which follows from the Lemma \ref{lm} but on the other hand the left hand side of equation \eqref{inf} is not differentiable at $t=1$. Therefore $S_\infty$ is not isometrically embedded in $S_p$.
	
	\vspace{.1in}
	\noindent
	\textbf{Case-IV:} For the case when $(q,p)\in[2,\infty)\times\{\infty\}$ one can prove it by the same argument as in Case 1. of Theorem \ref{isometric}.
\end{proof}
\begin{rmrk} It follows from our results that for $(q,p)\in(1,\infty]\times(1,\infty) \cup(1,\infty)\setminus\{3\}\times\{1,\infty\}$ and $p\neq q,$ with $p\neq q,$ if there exists an isometric embedding of $L_q(\mathcal M)$ into $S_p^n$, where $\mathcal M$ is a hyperfinite von Neumann algebra with normal semifinite faithful trace, and $2\leq \text{dim}L_q(\mathcal M)\leq n,$ then we must have $q=2$. Moreover, for the infinite-dimensional case if $(q,p)\in \left(1,\infty\right)\times\left[2,\infty \right)\cup[4,\infty)\times\{1\} \cup\{\infty\}\times\left( 1,\infty\right)\cup[2,\infty)\times\{\infty\}$ be such that $p\neq q$ and there exists an isometric embedding of $L_q(\mathcal M)$ into $S_p$, where $\mathcal M$ is a hyperfinite von Neumann algebra with normal semifinite faithful trace, and $\text{dim}L_q(\mathcal{M})\geq 2,$ then, we must have $q=2$.
\end{rmrk}	
%\begin{rmrk}
%It might be interesting to ask which von Neumann algebras have the property that whenever $L_q(\mathcal{M})$ embeds isometrically into $L_p(\mathcal M)$, the $q$ must be $2$? Let $\mathcal M$ be a von Neumann algebra of type $I$ with separable pre-dual. In this case, it is well-known that $\mathcal M$ is Jordan $*$-isomorphic to a countable $\ell_\infty$-direct sum of von Neumann algebras of the form $\mathcal A_n\otimes M_n$ and $\mathcal A\otimes \mathcal{B}(\hil),$ where $\mathcal A_n$ and $\mathcal A$ are commutative von Neumann algebra with separable preduals, and $\mathcal H$ is separable and $\mathcal A_n$ finite-dimensional. As $L_p(0,1)$ contains $\ell_2^2$ isometrically,  $\mathcal A\otimes \mathcal{B}(\hil),$ cannot appear in the direct summand. Clearly, in this case, $\mathcal M$ is Jordan $*$-isomorphic to $\ell_\infty$-direct sum of $\mathcal A_n\otimes \mathcal{B}(\hil)$ which is $*$-isomorphic to a $B(\mathcal K)$ for some separable von Neumann algebra $\mathcal K$.
%\end{rmrk}
%\begin{rmrk}
%Let $\mathcal R$ be hyperfinite type $\text{II}_1$ factor. Then for $1<p<\infty$ and $2\leq q<\infty$ no linear map $T:\mathcal R\to\mathcal R$ extends to an isometry from $L_q(\mathcal R)$ to $L_p(\mathcal R).$ 
%\end{rmrk}

\section*{Acknowledgements} We thank Prof. Gadadhar Misra for several useful pieces of advice. The first named author thankfully acknowledges the financial support provided by Mathematical Research Impact Centric Support (MATRICS) grant, File no: MTR/2019/000640, by the Science and Engineering Research Board (SERB), Department of Science and Technology (DST), Government of India. The second named author is supported by NSF of China ($\#$12071355). The third named author acknowledges Serb-DST(ECR/2018/002243 dated 12 March 2019) for financial support. The fourth named author thanks the Indian Institute of Technology Guwahati, Government of India, for the financial support. The fifth named author acknowledges Key-subsidy postdoctoral fellowship supported by Wuhan University (NSF (China), Grant-12071355). We sincerely thank the referee for several suggestions which considerably improved the presentation of this paper.

%\affiliationone{% in this example, two authors share an institution
%	Arup Chattopadhyay\\
%	Department of Mathematics, Indian\\
%	Institute of Technology Guwahati\\
%	Guwahati, 781039, India.
%	\email{arupchatt@iitg.ac.in\\ 2003arupchattopadhyay@gmail.com}}
% Important: Do not put any empty line here.
%\affiliationtwo{% in this example, one author has two addresses}
%	Guixiang Hong\\
%	School of Mathematics and Statistics\\
%	Wuhan university, Wuhan-430072, China.
%	\email{guixiang.hong@whu.edu.cn}}
% Important: Do not put any empty line here.
% Use \affiliationthree{} for any address positioned under \affiliationone
% Use \affiliationfour{}  for any address positioned under \affiliationtwo
%\affiliationthree{
%	Avijit Pal\\
%	Department of Mathematics, Indian\\ Institute of Technology, Bhilai-492015,\\ India.
%	\email{avijit@iitbhilai.ac.in}
%} %inserts a space to make this field empty
%\affiliationfour{%
%	Chandan Pradhan\\
%	Department of Mathematics, Indian\\
%	Institute of Technology Guwahati\\
%	Guwahati, 781039, India.
%	\email{chandan.math@iitg.ac.in,\\chandan.pradhan2108@gmail.com}}

%\affiliationone{%
%	Samya Kumar Ray\\
%	School of Mathematics and Statistics, \\Wuhan University, Wuhan-430072, China.
%	\email{samyaray7777@gmail.com}}
%
\end{document}